\documentclass[
	11pt,oneside,british,a4paper]{amsart}
	\usepackage[
  margin=2cm,
  includefoot,
  footskip=1in,
]{geometry}

\usepackage{amsmath}
\usepackage{amssymb}
\usepackage{amsfonts}

\usepackage{amsrefs}

\usepackage{hyperref}
\usepackage{amsthm}
\usepackage{enumerate}
\usepackage[mathscr]{eucal}
\usepackage[all]{xy}
\usepackage{graphicx}
\usepackage{caption}
\usepackage{colortbl}

\usepackage{tikz}
\usepackage{bbm}

\usepackage{verbatim}

       \newcommand{\ttfl} {\tfrac{2^{-2\ell}\ux-\uy}{t}}

\newtheorem{thm}{Theorem}[section]
\newtheorem{lemma}[thm]{Lemma}
\newtheorem{corollary}[thm]{Corollary}
\newtheorem{proposition}[thm]{Proposition}

\numberwithin{equation}{section}
\theoremstyle{definition}
\newtheorem{rem}[thm]{Remark}
\theoremstyle{definition}
\newtheorem{rems}[thm]{Remarks}

\newcommand{\Be}{\begin{equation}}
\newcommand{\Ee}{\end{equation}}
\newcommand{\Bea}{\begin{align}}
\newcommand{\Eea}{\end{align}}
\newcommand{\Beas}{\begin{align*}}
\newcommand{\Eeas}{\end{endalign*}}
\newcommand{\Benu}{\begin{enumerate}}
\newcommand{\Eenu}{\end{enumerate}}
\newcommand{\Bi}{\begin{itemize}}
\newcommand{\Ei}{\end{itemize}}

\newcommand{\np}{\text{NP}}
\newcommand{\eq}{\text{Eq}}
\newcommand{\im}{\text{IM}}

\newcommand\bbone{{\mathbbm 1}}

\def\th2{{\theta_{2n}}}
\def\th3{\overline{\theta}}

\def\uw{{\underline w}}

\def\intslash{\rlap{\kern  .32em $\mspace {.5mu}\backslash$ }\int}
\def\qsl{{\rlap{\kern  .32em $\mspace {.5mu}\backslash$ }\int_{Q_x}}}

\def\R{\mathbb R}

\def\emph#1{{\it #1 }}

\def\rank{{\text{\rm rank }}}

\def\supp{{\text{\rm supp }}}

\def\inn#1#2{\langle#1,#2\rangle}

\def\lc{\lesssim}
\def\gc{\gtrsim}

             \def\La{\Lambda}

\def\om{\omega}

\def\fI{{\mathfrak {I}}}
\def\fJ{{\mathfrak {J}}}

\def\fM{{\mathfrak {M}}}

\def\fS{{\mathfrak {S}}}

\def\fk{{\mathfrak {k}}}

\def\fn{{\mathfrak {n}}}

\def\fx{{\mathfrak {x}}}

\def\bbC{{\mathbb {C}}}

\def\bbH{{\mathbb {H}}}

\def\bbN{{\mathbb {N}}}

\def\bbR{{\mathbb {R}}}

\def\bbZ{{\mathbb {Z}}}

\def\cJ{{\mathcal {J}}}
\def\cK{{\mathcal {K}}}

\def\cM{{\mathcal {M}}}

\def\cR{{\mathcal {R}}}
\def\cS{{\mathcal {S}}}
\def\cT{{\mathcal {T}}}
\def\cU{{\mathcal {U}}}

\def\cZ{{\mathcal {Z}}}

\def\be#1{\begin{equation}\label{#1}}
\def\endeq{\end{equation}}
\def\endal{\end{align}}
\def\bas{\begin{align*}}
\def\eas{\end{align*}}
\def\bi{\begin{itemize}}
\def\ei{\end{itemize}}

\usepackage{stackengine}
\newcommand{\ubar}[1]{\stackunder[1.2pt]{$#1$}{\rule{.9ex}{.075ex}}}

\newcommand{\oa}{\bar{\alpha}}
\newcommand{\ua}{\ubar{\alpha}}
\newcommand{\ox}{{\bar x}}
\newcommand{\oy}{{\bar y}}
\newcommand{\ux}{{\ubar x}}
\newcommand{\uy}{{\ubar y}}
\newcommand{\uom}{{\ubar \omega}}
\newcommand{\oom}{{\bar \omega}}
\newcommand{\oH}{\bar{H}}

\begin{document}
\title
[On the Kor\'anyi Spherical  maximal function on Heisenberg groups]{On the Kor\'anyi Spherical  maximal function on Heisenberg groups}
\author[ R. Srivastava]{Rajula Srivastava}
\address{Rajula Srivastava: Department of Mathematics, University of Wisconsin, 480 Lincoln Drive, Madison, WI, 53706, USA.}
\email{rsrivastava9@wisc.edu}
\maketitle
\begin{abstract}
We prove $L^p\to L^q$ estimates for the local maximal operator associated with dilates of the K\'oranyi sphere in Heisenberg groups. These estimates are sharp up to endpoints and imply new bounds on sparse domination for the corresponding global maximal operator. We also prove sharp $L^p\to L^q$ estimates for spherical means over the Kor\'anyi sphere, which can be used to improve the sparse domination bounds in \cites{GangulyThangavelu} for the associated lacunary maximal operator. 
\end{abstract}
\section{Introduction}
Let $\bbH^n=\bbR^{2n} \times \bbR$ be the Heisenberg group of real Euclidean dimension $2n+1$. We shall use the notation $x=(\ux, \ox)$, $y=(\uy, \oy)$, with $\ux, \uy\in \bbR^{2n}$ and $\ox, \oy\in\bbR$, to denote elements of $\bbH^n$. The group law is given by 
\[x \cdot y =(\ux+\uy, \ox+\oy + \tfrac{1}{2}\ux^\intercal J\uy), \]
where \[J:=\begin{pmatrix}
0 & & I_n\\
-I_n &  & 0
\end{pmatrix}\]
is the $2n\times 2n$ standard symplectic matrix (and $I_n$ is the $n\times n$ identity matrix).
Further, the Kor\'anyi norm of an element $x$ is defined to be \[|x|_K:=(|\ux|^4+|\ox|^2)^{\frac{1}{4}}.\] This norm is homogeneous of degree one with respect to the natural parabolic dilation structure $\delta_t((\ux,\ox)):=(t\ux,t^2 \ox)$ on $\bbH^n$. Let $S_K$ be the sphere centred at the origin, of radius one with respect to the Kor\'anyi norm.  
There exists a unique Radon measure $\mu$ on $S_K$ induced by the Haar measure on $\bbH^n$. 
For $t>0$, the averaging operator associated with the $\mu_t$, the $t$-dilate of this measure, is given by
\begin{equation}
    \label{eq:koravg}
    \mathcal{A}_tf(x):=f*\mu_t(x)=\int_{S_K} f(\ux-t\ubar{\omega},\ox-t^2\oom-\tfrac{t}{2} 
\ux^\intercal J\ubar{\omega} ))\,d\mu((\ubar{\omega},\oom))
\end{equation}
and let $\mathcal{A}f(x,t):=\mathcal{A}_tf(x)$.
\begin{figure}
    \centering
    \captionsetup{justification=centering}
    \includegraphics[width=.4\linewidth,height=.4\textheight,keepaspectratio]{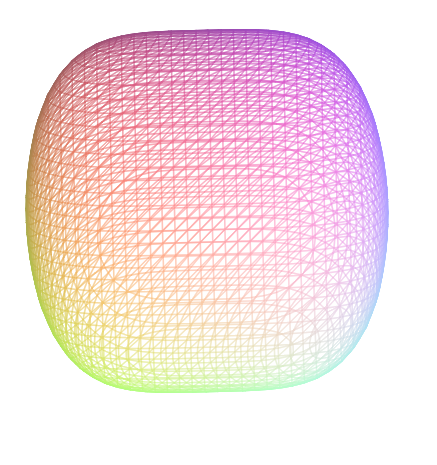}
    \caption{The Kor\'anyi sphere centred at the origin in $\bbH^1$}
    \label{fig:Koranyi}
\end{figure}

In 1981, Cowling \cites{Cowling1980} considered the global Kor\'anyi maximal function \[\mathfrak{M} f(x):=\sup_{t>0}|\mathcal{A}_t f(x)|.\] He showed that $\mathfrak{M}$ is bounded on $L^p(\bbH^n)$ for all $p>\frac{2n+1}{2n}$ and $n\geq 2$.  More recently, in \cite{GangulyThangavelu}, Ganguly and Thangavelu considered the lacunary variant \[\mathfrak{M}^{\text{lac}}f(x):=\sup_{k>0}|\mathcal{A}_{2^k} f(x)|\] on $\bbH^n$ for $n\geq 2$ and proved that $\mathfrak{M}^{\text{lac}}$ is bounded on $L^p$ for all $1<p<\infty$. They did so by establishing a $(p,q')$-sparse domination result for $\mathfrak{M}^{\text{lac}}$. Closely following Lacey's approach in the Euclidean case \cites{laceyJdA19}, they made use of an induction argument relying on $L^p\to L^q$ estimates for spherical means over the unit Kor\'anyi sphere \[\mathcal{A}_1f(x)=f*\mu(x),\] a problem which is also of interest in its own right.  
The objective  is to find the best possible value of $q$ in such an estimate. 
In \cite{GangulyThangavelu} it is proved that $\mathcal{A}_1$ maps $L^p(\bbH^n)$ to $L^q(\bbH^n)$ provided that $(\tfrac 1p, \tfrac 1q)$ belongs to the interior of the triangle with corners $(0,0)$, $(1,1)$, $(\tfrac{2n}{2n+1}, \frac 1{2n+1})$, as well as the line joining the points $(0,0)$ and $(1,1)$. In the following theorem, we provide $L^p\to L^q$ bounds that are sharp.
\begin{thm}
\label{thm single avg}
Let $n\geq 1$. The inequality 
\begin{equation*}
\label{Hsphmeans} \|f*\mu\|_{L^q(\bbH^n)} \lc \|f\|_{L^p(\bbH^n) }  
\end{equation*}  holds for all $f\in  L^p(\bbH^n)$ if and only if  
$(\tfrac 1p, \tfrac 1q)$ belongs to the closed triangle  with corners 
$(0,0)$, $(1,1)$ and $(\frac{2n+1}{2n+2}, \frac{1}{2n+2}).$
\end{thm}

It is even more interesting to consider the local maximal function 
\[Mf(x):=\sup_{t\in [1,2]}|\mathcal{A}_t f(x)|.\] In light of the recent $L^p\to L^q$ estimates for the local maximal operator associated with codimension two spheres in the Heisenberg groups \cites{BagchiHaitRoncalThangavelu, roos2021lebesgue} (see the following remarks), it is natural to seek similar estimates for $M$. Such an estimate would also imply sparse bounds for the global maximal operator $\mathfrak{M}$.
Our main theorem contains $L^p\to L^q$ estimates for $\mathcal{M}$ which are sharp up to endpoints.
\begin{thm} \label{thm:max full} Let $n\ge 2$.
Let  $\cR$ be the closed quadrilateral  with corners
\begin{equation*}\label{quadrilateral}\begin{gathered}
Q_1=(0,0), \qquad Q_2=(\tfrac{2n}{2n+1}, 
\tfrac{2n}{2n+1}),  
\\
Q_3=(\tfrac{2n}{2n+1},\tfrac{1}{2n+1}), \quad 
 Q_4=\left(\tfrac{n(2n+1)}{2n^2+2n+2},\tfrac{n}{2n^2+2n+2}\right).
\end{gathered} 
\end{equation*} Then 
 
(i) $M:L^p(\bbH^n)\to L^q(\bbH^n)$ is bounded if $(\frac 1p, \frac 1q)$ belongs to the interior of $\cR$, or to the open boundary  segment $(Q_2,Q_3)$, or to the half open boundary segment $[Q_1,Q_2)$. 

(ii) $M$ does not map $L^p(\bbH^n)$ to $L^q(\bbH^n)$  if $(\frac 1p, \frac 1q)\notin \cR$.

(iii) $M$ does not map $L^p(\bbH^n)$ to $L^p(\bbH^n)$   for  $(\tfrac 1p,\tfrac 1q)=Q_2$. 
\end{thm}
\begin{figure}
    \centering
    \begin{tikzpicture}[scale=5]
\draw[step=1cm,thick,->] (0,0) -- (1,0) node[anchor=north west] {$\frac{1}{p}$};
\draw[step=1cm,thick,<->] (0,0) -- (0,1) node[anchor=south east] {$\frac{1}{q}$};
\draw[opacity=.8] (0,0) -- ({4/5},{4/5})--({4/5},{1/5})--({5/7},{1/7})--(0,0); 
\fill [color=gray,opacity=.3] ({4/5},{4/5})--({4/5},{1/5})--({5/7},{1/7})--(0,0);
\draw [dashed,opacity=.3] (1,0) -- (0,1);
		\draw [dashed,opacity=.3] (0,0) -- (1,{1/5}); 
		\draw [dashed,opacity=.3] (0,0) -- (1,1); 
		\draw [dashed,opacity=.2]
		(.5, 0) -- (.5, .5);
\fill (0,0) node [left] {$Q_1$} circle [radius=.02em];
		\fill ({4/5},{4/5}) node [above left] {$Q_{2}$} circle [radius=.02em];
		\fill ({4/5},{1/5}) node [right] {$Q_{3}$} circle [radius=.02em];
		\fill ({5/7},{1/7}) node [below right] {$Q_{4}$} circle [radius=.02em];		
		
\end{tikzpicture}
    \caption{The region $\mathcal{R}$ in Theorem \ref{thm:max full}, for $n=2$.}
    \label{fig:my_label}
\end{figure}
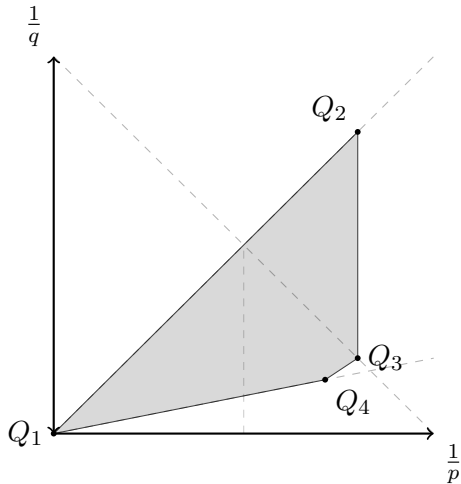
There has been considerable progress in the recent years on the problem of establishing $L^p$ improving properties of localized maximal functions associated with surfaces, both in the Euclidean and Heisenberg settings. A few remarks are in order to shed light on the chief features of our problem which make it different from some related bodies of work.
\begin{rems}
(i) \textit{Non-vanishing Rotational Curvature:} A different maximal function on the Heisenberg group, considered in \cites{NevoThangavelu1997, MuellerSeeger2004, NarayananThangavelu2004, BagchiHaitRoncalThangavelu, roos2021lebesgue, BeltranGuoHickmanSeeger}, is associated to averages over codimension two spheres contained in a dilation invariant subspace of $\bbH^n$. In contrast, the averaging operator $\mathcal{A}_t$ in \eqref{eq:koravg} is associated to a surface of codimension one which respects the non-isotropic dilation structure on $\bbH^n$. Further, unlike the codimension two case, the oscillatory integral operators associated with $\mathcal{A}_t$ possess non-vanishing rotational curvature at fixed time $t$
(this fact 
is also implicit in \cite{phong1986hilbert} and proven directly in \cite{schmidt1998maximaloperatoren}). In this sense, $\mathcal{A}_t$ could be considered to be a faithful analog of the Euclidean spherical average in the Heisenberg setting.

(ii) \textit{Vanishing of Cinematic Curvature:} Unlike the Euclidean spherical averaging operator, the rank of the cinematic curvature matrix \eqref{curvmatrix} associated with 
$\mathcal{A}$  
is zero at the north and south poles of the Kor\'anyi sphere. 
To deal with the issue of flatness (that is, the absence of cinematic curvature in the sense of \cite{sogge1991propagation}) at the poles, we need to employ a scaling argument. The first step in such an argument is to dyadically decompose the relevant kernel based on the distance from the poles, followed by rescaling each piece to a region where the curvature does not vanish. This idea was used by Iosevich \cites{iosevich1994maximal} to establish the $L^p$ boundedness of global maximal operators associated to families of flat, finite type curves in $\bbR^2$. It was also used in \cites{manna2017pl} and more recently \cites{LWZ} to prove $L^p\to L^q $ estimates for local maximal functions along some finite type curves in $\bbR^2$ and hypersurfaces in $\bbR^3$. However, unlike  \cites{iosevich1994maximal,LWZ,manna2017pl}, our problem is based in a non-Euclidean setting. The Heisenberg group structure calls for a new type of scaling argument.
\end{rems}

\subsection*{Comparison with the Euclidean convolution structure} The Heisenberg group structure plays a crucial role in our analysis. It shows itself in the presence of the bilinear term involving the symplectic matrix $J$ in the defining equation of the Kor\'anyi sphere centred at $x$ and of radius $t$, given by
\begin{equation}
    \label{eq Heis Kor sph}
    |\ux-\uy|^4+|\ox-\oy+\tfrac{1}{2}\ux^{\intercal}J\uy|^2=t^4.
\end{equation}
To illustrate its significance, we contrast our situation with that of the Kor\'anyi sphere in $\bbR^{2n+1}$ when the translations are Euclidean and not given by the Heisenberg law. The defining equation of such a sphere centred at $x$ and of radius $t$ is given by
\begin{equation}
    \label{eq Eucl Kor sph}
    |\ux-\uy|^4+|\ox-\oy|^2=t^4.
\end{equation}
\begin{figure}
    \centering
    \begin{tikzpicture}[scale=5]
\draw[step=1cm,thick,->] (0,0) -- (1,0) node[anchor=north west] {$\frac{1}{p}$};
\draw[step=1cm,thick,<->] (0,0) -- (0,1) node[anchor=south east] {$\frac{1}{q}$};
\draw[opacity=.8] (0,0) -- ({4/5},{4/5})--({4/5},{1/5})--({5/7},{1/7})--(0,0);
\draw[color=blue,opacity=1,dashed]
({2/3},{1/3})--({4/5},{3/5});
\draw[color=blue,opacity=1,dashed]
({5/8},{1/4})--({5/9},{1/9});
({4/5},{1/5})--({10/13},{2/13})--(5/7,1/7);
\fill [color=gray,opacity=.3] ({4/5},{4/5})--({4/5},{1/5})--({5/7},{1/7})--(0,0);
\fill [color=blue,opacity=.3] ({2/3},{2/3})--({2/3},{1/3})--({5/8},{1/4})--(0,0);
\draw [dashed,opacity=.3] (1,0) -- (0,1);
		\draw [dashed,opacity=.3] (0,0) -- (1,{1/5}); 
		\draw [dashed,opacity=.3] (0,0) -- (1,1); 
		\draw [dashed,opacity=.2]
		(.5, 0) -- (.5, .5);
\fill (0,0) node [left] {$Q_1=P_1$} circle [radius=.02em];
        \fill ({4/5},{4/5}) node [above left] {$Q_{2}$} circle [radius=.02em];
		\fill ({4/5},{1/5}) node [right] {$Q_{3}$} circle [radius=.02em];
		\fill ({5/7},{1/7}) node [below left] {$Q_{4}$} circle [radius=.02em];
		\fill ({2/3},{2/3}) node [above left] {$P_{2}$} circle [radius=.02em];
		\fill ({2/3},{1/3}) node [right] {$P_{3}$} circle [radius=.02em];
		\fill ({5/8},{1/4}) node [below right] {$P_{4}$} circle [radius=.02em];
		\fill ({4/5},{3/5}) node [below right] {$P_{3}'$} circle [radius=.02em];
		\fill ({5/9},{1/9}) node [below left] {$P_{4}'$} circle [radius=.02em];
		
\end{tikzpicture}
    \caption{A comparison between the regions of $L^p\to L^q$ boundedness for the Kor\'anyi spherical maximal operator corresponding to Heisenberg convolution (in grey) versus the Euclidean convolution (in blue), for $n=2$.}
    \label{fig comparison}
\end{figure}
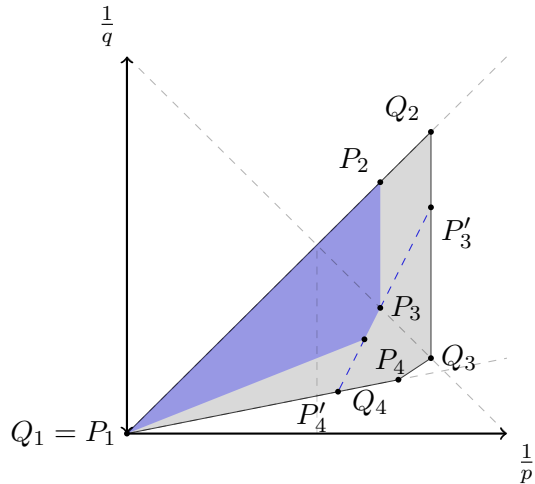  
The quadrilateral $\mathcal{R}=Q_1Q_2Q_3Q_4$ in Figure \ref{fig comparison} (in grey) depicts the region of almost sharp $L^p\to L^q$ estimates for the Kor\'anyi spherical maximal operator with Heisenberg group structure (given by Theorem \ref{thm:max full}). Since the necessary conditions for this operator (see \S \ref{sec counter eg}) also apply to the Kor\'anyi spherical maximal operator described with respect to Euclidean convolution, the region corresponding to the sharp $L^p\to L^q$ estimates for the Euclidean version is contained in $\mathcal{R}$. In fact, this containment is proper, as can be seen by means of a standard Knapp example adapted to a  homogeneous Euclidean hypersurface of degree four, which yields the sharpness of the edge $P_3'P_4'$ in the interior of $\mathcal{R}$. The quadrilateral $\mathcal{P}=P_1P_2P_3P_4$ (in blue) depicts the region in which scaling type arguments can be used to obtain a positive result for the Kor\'anyi maximal operator in the Euclidean setting. Its coordinates in $\bbR^{2n+1}$ are given by
\[\begin{gathered}
P_1=(0,0), \qquad P_2=(\tfrac{n}{n+1}, 
\tfrac{n}{n+1}),  
\\
P_3=(\tfrac{n}{n+1},\tfrac{1}{n+1}), \quad 
P_4=\left(\tfrac{n(2n+1)}{2n^2+3n+2},\tfrac{2n}{2n^2+3n+2}\right).
\end{gathered}\]
The estimate at $P_2$ is the same as the $L^p$ bound contained in \cite[Theorem 3.1]{cowling1987oscillatory} for global maximal operators associated to averages over Euclidean hypersurfaces which are graphs of smooth homogeneous functions with non-vanishing Hessian away from the origin, with the degree of homogeneity being at most twice the dimension of the surface. The estimates corresponding to $P_3$ and $P_4$ can be obtained by following the same arguments as in the proof of Theorem \ref{thm:max full} (but in a more straightforward manner). Eschewing detailed calculations, which shall appear in the author's Ph.D. thesis \cite{SrivThesis}, we now briefly discuss the reasons behind these differences between the Heisenberg and Euclidean settings.

In the Heisenberg case, the term $\ux^{\intercal}J\uy$ changes the geometry of the Kor\'anyi sphere as it is translated away from the origin and is responsible for the non-vanishing of the rotational curvature, even at the poles. 
As a result, H\"ormander's classical $L^2$ theory \cite[Ch. IX.1]{Stein-harmonic} can be used to establish a sharp fixed time $L^2$ estimate. 
This in turn implies the positive result in Theorem \ref{thm single avg} after interpolation with easily obtainable estimates involving the $L^1$ and $L^\infty$ spaces. Using standard Sobolev embedding, the estimates at the vertices $Q_1, Q_2, Q_3$ of the quadrilateral $\mathcal{R}$ also follow. They are sharp up to endpoints and analogous to the corresponding estimates for the standard sphere in $\bbR^{2n+1}$ defined using the Euclidean norm. 

This is in sharp contrast to the Kor\'anyi sphere with Euclidean translational structure, where the rotational curvature vanishes at the poles and a scaling argument (along the lines of \cites{iosevich1994maximal}) is needed. Further, the fixed time $L^2$ estimate obtained using scaling is worse than the same estimate for the standard Euclidean sphere, and consequently, so are the corresponding diagonal and off-diagonal estimates.

As mentioned already, the Heisenberg estimate at $Q_4$ is indeed affected by the flatness at the poles, as is the corresponding estimate at $P_4$ for the Kor\'anyi sphere in the Euclidean case. In both cases, after suitable decomposition and scaling, we need to prove a fixed time $L^2$ estimate for each rescaled piece separately. 
However, in order to establish 
these estimates in the Heisenberg case,  we need to introduce further localization and make use of an almost $L^2$ orthogonality that exists between the different localized pieces (see \S\ref{subsubsec almost L2}). This is due to the surprising feature that the non-isotropic scaling in the ``input variable'' $y$ is different from the one which shows up in the ``output variable'' $x$, with $x, y$ in \eqref{eq Heis Kor sph}. This phenomenon is absent in the Euclidean version and occurs, again, due to the presence of the bilinear term $\ux^\intercal J\uy$ in the second term in \eqref{eq Heis Kor sph}  (see \S\ref{sec scaling}, Remark \ref{rem diff scal}).  

Quite pleasantly, this ``imbalanced'' scaling argument yields a better estimate at $Q_4$ in the Heisenberg case than the estimate at $P_4$ for the Kor\'anyi sphere with Euclidean convolution structure. The former estimate is consistent with a new Knapp-type counterexample \S\ref{sec:Q3Q4} which establishes the sharpness of the edge $Q_3Q_4$ in the Heisenberg case. As mentioned already, the edge $P_3P_4$ for the Kor\'anyi sphere in the Euclidean setting is also sharp, which can be seen by using a standard Knapp example adapted to a Euclidean degree four homogeneous surface.

The estimates at $Q_4$ (and $P_4$) also rely on an $L^\infty$ bound on the kernel of an associated oscillatory integral operator. To do so, we shall seek to make use of estimates for standard oscillatory integrals of Carleson-Sj\"olin-H\"ormander type, in particular a variant of  Stein's theorem  \cite{SteinBeijing} formulated in \cite{MSS93} which relies on the maximal possible number of nonvanishing curvatures for a cone in the fibers of the canonical relation (see \eqref{eq:rank-curv}). 

The above discussion was focused on the estimates around the poles. Both the curvature conditions also hold at the equator and in the intermediate region of the Kor\'anyi sphere, enabling us to apply H\"ormander's $L^2$ estimate and the aforementioned Stein's theorem directly. However, their verification is still technically involved in the Heisenberg case due to the additional bilinear term. 
\subsection*{Plan of the paper}
The paper is organized as follows:
\begin{itemize}
\item \S \ref{sec prelim} contains a localization argument for the function $f$, and a partition of unity argument which streamlines our analysis into three regimes, based on whether the kernel of the Fourier integral operator being considered is supported around the poles, at the equator or in an intermediate region. We also describe the corresponding parametrization of the Kor\'anyi sphere (explicitly around the poles and implicitly in the other two regimes). 
\item In \S \ref{sec: main estimates} we list the basic estimates 
and reduce them to corresponding ones for oscillatory integral operators. 
\item In \S \ref{sec scaling} we use a scaling argument to understand the behaviour of our operator around the poles (or the north pole to be more specific, but the behaviour around the south pole is exactly the same due to symmetry). 
\item \S \ref{sec stein tomas} describes a Stein-Tomas type argument to reduce the estimate at $Q_4$ to an $L^2$ estimate and an $L^\infty$ bound on the kernel of an oscillatory integral operator. 
\item \S \ref{sec matrix inv} contains a few technical results about matrices and radial functions which will be crucial to calculations verifying the non-vanishing of the rotation curvature and the additional curvature condition \eqref{eq:rank-curv}. 
\item In \S \ref{sec np}, we verify these two conditions around the poles (for the rescaled operators). \S\ref{subsubsec almost L2} also contains an almost orthogonality argument crucial to the proof of the fixed time $L^2$ estimate for the rescaled operators. 
\item In \S \ref{sec eq im}, we verify the curvature conditions at the equator and in the intermediate region, thus finishing the proofs of the positive results contained in Theorems \ref{thm single avg} and \ref{thm:max full}. 
\item In \S \ref{sec counter eg} we provide counterexamples to show the sharpness of these theorems. 
\item In \S\ref{sec further} we briefly discuss their implications on sparse bounds for the lacunary and global maximal functions. 
\item \S \ref{section trans invar} contains a technical lemma about the invariance of both the curvature conditions under the Heisenberg group law.
\end{itemize}

\subsection*{Notation} 
Partial derivatives and tangent vectors along the coordinate direction $e_j$ will be denoted by the subscript $j$ for $j\in \{1,2,\ldots, 2n+1\}$. For a function $f: \bbR^{2n}\to \bbR^{2n}$, $f''$ shall be used to denote its Hessian which is a $2n\times 2n$ matrix.
By $A\lesssim B$ we shall mean that $A\le C\cdot B,$ where $C$ is a positive constant and $A\approx B$ shall signify that $A\lesssim B$ and $B\lesssim A$. $A\lesssim_D B$ shall mean that $A\leq C\cdot B$ with the positive constant $C$ depending on the parameter $D$.

\subsection*{Acknowledgement} 
The author is grateful to her advisor, Andreas Seeger, for suggesting this problem and his invaluable advice at various stages of the project. She is also thankful to the Hausdorff Research Institute of Mathematics and the organizers of the trimester program “Harmonic Analysis and Analytic Number Theory” for a productive stay in the summer of 2021. Research
supported in part by NSF grants DMS 1764295 and DMS 2054220.
\section{Preliminaries}
\label{sec prelim}
\subsection{Localization} 
Without loss of generality, we can limit our consideration to functions $f$ supported in a small subset of a thin neighborhood of the Kor\'anyi sphere of radius one centred at the origin. To see this we use the group translation to tile $\bbH^n$. Let $B_0=[-\tfrac 12,\tfrac 12)^{2n+1}$ and, for $\fn\in \bbZ^{2n+1}$, let $B_{\fn}= \fn \cdot B_0$, i.e.
 $B_{\fn}=\{(\ubar \fn+\ubar z,\bar \fn+\bar z+\tfrac{1}{2}\ubar \fn J\ubar z): z\in B_0\}$.
 One then verifies that $\sum_{\fn\in \bbZ^{2n+1} }\bbone_{B_\fn}=1$.
 Moreover, for $t\in [1,2]$, the $t$-dilates of the measure $\mu$ are supported in $\{w\in \bbH^n: |\ubar w|\le 2, |\bar w| \le 4\}$, hence in the union of $B_\fk$ with 
 $|\fk_j|\le 2$ for $j\le 2n$ and $|\fk_{2n+1}|\le 6$. Denote this set of indices by $\fJ$. Then  \[\supp \big(\mathcal{A}_t[f\bbone_{ B_\fn}]\big) \subset 
 \bigcup_{\fk\in \fJ} (\fn\cdot B_0\cdot B_\fk) \subset \bigcup_{\tilde \fn\in \fI(\fn)} B_{\tilde \fn},
 \]
 where $\fI(\fn)$ is a  set of indices $\tilde\fn$ with $|\fn_j-\tilde \fn_j|\le C(J,n) $ for $j=1,\dots, 2n+1$. This consideration of spatial orthogonality allows us to reduce to the case of functions $f$ supported in the union of a finite number of Heisenberg tiles $B_{\fn}$ which cover the unit Kor\'anyi sphere and its $t$ dilates for $t\in[1,2]$. We can further partition these tiles (resp. the interval $[1,2]$) into a finite number of sub-tiles (resp. sub-interval) of small enough size and consider the averaging operator associated to each such piece separately, summing up the estimates at the end. This reduces matters to the case when $f$ is supported in a small subset of a thin neighborhood of the unit Kor\'anyi sphere, of length at most $\frac{2^{-400n}}{100n}$ in each coordinate direction. As a consequence, we can also assume that $\mathcal{A}_tf$ is supported in a small neighborhood of the origin of length at most $\frac{2^{-400n}}{10n}$ in each direction.

We shall frequently make use of both Taylor approximation and the implicit function theorem to express the phase function of oscillatory integral operators. Both these theorems require the corresponding amplitude to be supported in a set of constant but sufficiently small size. By choosing a suitable smooth partition of unity, we can express $\mathcal{A}_t$ as a finite sum of operators, each of which is localized to a small region of the Kor\'anyi sphere. It then suffices to prove the required estimates for each such operator independently. We now explicitly describe such a localization.

Given a point $w\in S_K$, let $\mathcal{Q}(w):=\{\Tilde{w}\in S_K:|w-\Tilde{w}|<2^{-400n}\}$. 
Since the Kor\'anyi sphere $S_K$ is a compact submanifold of $\bbH^n$, there exists a finite set of points $\{w_{\nu}\}_{\nu}\subseteq S_K$ such that the collection $\{\mathcal{Q}_{\nu}:=\mathcal{Q}(w_{\nu})\}_{\nu}$ forms an open cover of $S_K$. Let $\sum_\nu \eta_\nu$ be a smooth partition of unity subordinate to $\{\mathcal{Q}_{\nu}\}_{\nu}$. Define $\mu^{\nu}:=\mu \eta_{\nu}$ so that $\mu=\sum_\nu \mu^{\nu}$.

We can rewrite \eqref{eq:koravg} as \[\mathcal{A}_tf(x):=\sum_\nu\int_{\mathcal{Q}_\nu} f(\ux-t\ubar{\omega},\ox-t^2\oom-\tfrac{t}{2} 
\ux^\intercal J\ubar{\omega} )\,d\mu^\nu(\ubar{\omega},\oom).\]
Using a change of variables $\uy=\ux-t\uom$, $\oy=\ox-t^2\oom-\tfrac{t}{2}
\ux^\intercal J\ubar{\omega}$, we can express
\[\mathcal{A}_tf(x):=\sum_\nu\int_{\mathcal{Q}_\nu} f(\uy,\oy)\,d\mu^\nu\left(\frac{\ux-\uy}{t},\frac{\ox-\oy+\tfrac{1}{2}\ux^{\intercal}J\uy}{t^2}\right).\]

\subsection{Parametrization}
The defining equation of the Kor\'anyi sphere centred at $x$ and of radius $t$ is
\begin{equation}
    \label{eq defining}
    F(x,t,y):=|\ux-\uy|^4+\left|\ox-\oy+\tfrac{1}{2}\ux^{\intercal}J\uy\right|^2-t^4=0.
\end{equation}
Differentiating, we have
\begin{equation}
    \label{eq grad defining}
    \nabla_{x,t}F(x,t,y)=4|\ux-\uy|^2(\ux-\uy)+2\left(\ox-\oy+\tfrac{1}{2}\ux^{\intercal}J\uy\right)\left(\tfrac{1}{2}J\uy+e_{2n+1}\right)+4t^3e_{2n+2}.
\end{equation}
To parametrize the Kor\'anyi sphere as the graph of a smooth function locally, we need to express one of the coordinates of $y$ as a function of the other $y-$coordinates, $x$ and $t$. The coordinate we choose, and hence the parametrization, will depend upon the localization neighborhood $\mathcal{Q}_\nu$. 

Since the Kor\'anyi sphere is symmetric about the equator, we can limit our attention to its northern hemisphere. We will choose different parametrizations near the equator and around the north pole. Unlike in the case of the Euclidean sphere, we shall see that the curvature properties depend upon the neighborhood $\mathcal{Q}_\nu$ being considered. Heuristically speaking, due to the non-isotropic dilation structure of the Kor\'anyi sphere, the region around the equator exhibits similar curvature properties as the Euclidean sphere. However, around the north pole, the Kor\'anyi sphere behaves like the surface $\oy=1+|\uy|^4$, with the (cinematic) curvature vanishing at the pole. Other than these two extreme cases, the intermediate region needs to be carefully considered as well. 
\begin{rem}
The above discussion only applies to the cinematic curvature condition (see \eqref{eq:rank-curv}). The rotational curvature does not see these different parametrizations, and remains non-vanishing throughout. This has been observed before in the context of the Kor\'anyi sphere in \cites{schmidt1998maximaloperatoren} and is also to be expected in view of the $L^p$ bounds on the global maximal operator in \cites{Cowling1980}.
\end{rem}

Let $\mathcal{J}$ denote the indexing set of the cover $\{\mathcal{Q}_\nu\}_\nu$. Keeping the above discussion in mind, we partition $\mathcal{J}$ into three disjoint sets
\[\mathcal{J}=\mathcal{J}_{\text{Eq}}\cup \mathcal{J}_{\text{IM}}\cup \mathcal{J}_{\text{NP}},\]
where
\begin{align*}
\mathcal{J}_{\text{NP}}&=\{\nu: e_{2n+1}\in \mathcal{Q}_\nu\}\\
\mathcal{J}_{\text{Eq}}&=\{\nu: \text{ there exists } x=(\ux,\ox)\in \mathcal{Q}_\nu \text{ with } \ox=0\},\\
\mathcal{J}_{\text{IM}}&=\mathcal{J}\setminus(\mathcal{J}_{\text{Eq}}\cup \mathcal{J}_{\text{NP}}).
\end{align*}
\subsubsection*{Near the North Pole}
For $\nu\in \mathcal{J}_{\text{NP}}$ and for all $\left(\frac{\ux-\uy}{t},\frac{\ox-\oy+\tfrac{1}{2}\ux^{\intercal}J\uy}{t^2}\right)\in \mathcal{Q_\nu}$, we have $t^{-1}|\ux-\uy|<2^{-400n}$ and we can write
\begin{equation}
    \label{eq north pole form}
     \overline{x}-\overline{y} + \frac{1}{2}\underline{x}^T J \underline{y}=t^2\sqrt{1-\frac{|\ux-\uy|^4}{t^4}}
\end{equation}
Using Taylor's remainder formula, we can express 
\begin{equation}
    \label{eq G Taylor exp}
    \sqrt{1-|\uw|^4}= 1-\tfrac{|\uw|^4}{2}+|\uw|^8R(\uw),
\end{equation}
where $R$ is a smooth, bounded and radial function such that
\begin{equation}
    \label{eq Gerror}
    \left|\frac{\partial^\beta}{\partial \uw^{\beta}}R(\uw)\right|\leq 100n\,\, \text{ for } |\beta|\in\{0,1,2,3\} \text{ and } |\uw|<\frac{1}{2}.
\end{equation} 
Plugging \eqref{eq G Taylor exp} with $\uw=\frac{\ux-\uy}{t}$ into \eqref{eq north pole form}, we obtain
\begin{equation}
    \label{eq np form 1}
     \overline{y}=\mathcal{G}(x,t,\uy):=\overline{x} + \frac{1}{2}\underline{x}^T J \underline{y}-t^2+t^2\frac{|\ux-\uy|^4}{2t^4}+\frac{|\ux-\uy|^8}{t^8}R(\frac{\ux-\uy}{t}).
\end{equation}
\subsubsection*{Near the Equator}
Observe that for $\mathfrak{A}\in O(2n)$, the action $(\ux,\ox,t,\uy,\oy)\to (\mathfrak{A}\ux,\ox,t,\mathfrak{A}\uy,\oy)$ leaves the first and third terms in \eqref{eq defining} unchanged, and replaces $J$ by $J'=\mathfrak{A}^{\intercal}J\mathfrak{A}$ in the second term. Since $J'$ is still a skew-symmetric matrix with $J'^2=-I_{2n}$, as long as our analysis is uniform in the set of all matrices satisfying these two properties, using horizontal rotations, it suffices to consider only those $\nu\in \mathcal{J}_{\text{Eq}}$ such that the corresponding neighborhood $\mathcal{Q}_\nu$ contains $e_1$. Let $ \mathcal{J}_{\text{Eq}}'$ denote the collection of all such $\nu$.

For any $\left(\frac{\ux-\uy}{t},\frac{\ox-\oy+\tfrac{1}{2}\ux^{\intercal}J\uy}{t^2}\right)$  in such a neighborhood $\mathcal{Q}_{\nu}$, $t^{-1}|x_1-y_1|>1-2^{-400n}$ and $t^{-2}|\ox-\oy+\tfrac{1}{2}\ux^{\intercal}J\uy|<2^{-400n}$, and in view of \eqref{eq grad defining}, this implies that $\left|\frac{\partial F}{\partial y_1}\right|>t(1-2^{-400n})$. Write $y=(\uy,\oy)=(y_1,y',\oy)$ with $y'\in \bbR^{2n-1}$. We can use the implicit function theorem 
to express 
\begin{equation}
\label{eq equator form}
    y_1=H_1(x,t,y',\oy), \text{ with  }\nabla_{(x,t)}H_1(x,t,y',\oy)=-\left[\left(\frac{\partial F}{\partial y_1}\right)^{-1}\nabla_{x,t}F\right]\bigg|_{(x,t,H_1(x,t,y',\oy), y',\oy)}.
\end{equation}
\subsubsection*{Intermediate Region}
Again, using a horizontal rotation argument as above, it suffices to consider only those $\nu\in \mathcal{J}_{\text{IM}}$ such that the corresponding neighborhood $\mathcal{Q}_\nu$ contains $s_1e_1+s_{2n+1}e_{2n+1}$, with $\min(s_1, s_{2n+1})>2^{-200n}$ and $s_1^4+s_2^2=1$.  Let $ \mathcal{J}_{\text{IM}}'$ denote the collection of all such $\nu$.

For any $\left(\frac{\ux-\uy}{t},\frac{\ox-\oy+\tfrac{1}{2}\ux^{\intercal}J\uy}{t^2}\right)$  in such a neighborhood $\mathcal{Q}_{\nu}$, $\min\left(t^{-1}|x_1-y_1|, t^{-2}|\ox-\oy+\frac{1}{2}\ux^{\intercal}J\uy|\right)>2^{-400n}$ and $t^{-1}|x'-y'|<2^{-400n}$. In view of \eqref{eq grad defining}, this implies that $2^{-200n-1}<t^{-2}\left|\frac{\partial F}{\partial \oy}\right|<1$.  
We can use the implicit function theorem again, this time to express 
\begin{equation}
\label{eq intermediate form}
    \oy=\oH(x,t,\uy) \text{ with }\nabla_{(x,t)}\oH(x,t,\uy)=-\left[\left(\frac{\partial F}{\partial \oy}\right)^{-1}\nabla_{x,t}F\right]\bigg|_{(x,t, \uy,\oH(x,t,\uy))}.
\end{equation}

\section{Main Estimates}
\label{sec: main estimates}
Keeping the above parametrizations in mind, we are led to consider three types of generalized Radon transforms, associated to incidence relations given by the equations \eqref{eq np form 1}, \eqref{eq equator form} and \eqref{eq intermediate form},
with kernels supported in $\mathcal{Q}_\nu$ for $\nu$ in $\mathcal{J}_{\text{NP}}$, $\mathcal{J}_{\text{Eq}'}$ and $\mathcal{J}_{\text{IM}}'$. We choose $\nu_\np, \nu_\eq, \nu_\im$ in $\cJ_\np', \cJ_\eq'$ and $\cJ_\im'$ respectively, and fix them for the rest of the paper. For $\square\in \{\np, \eq, \im\}$, we define $\mu^\square:=\mu^{\nu_\square}$ and $\eta_\square:=\eta_{\nu_\square}$. 

Recall that we have reduced consideration to functions $f$ with small support around a thin neighborhood of the Kor\'anyi sphere, so that $\mathcal{A}_tf$ is supported around a small neighborhood of the origin. We can choose a smooth function $\chi_1$ supported in a slightly larger neighborhood of the origin so that $\chi_1(x)\mathcal{A}_tf(x)=\mathcal{A}_tf(x)$ for all $x\in \bbR^{2n+1}$. 
Let
\[
\Theta_t(x,y):=\left(\frac{\ux-\uy}{t},\frac{\ox-\oy+\tfrac{1}{2}\ux^{\intercal}J\uy}{t^2}\right).\]
We now define smooth cut-off functions $\Tilde{\chi}_\np(x,t,\uy):=\chi_1(x) \eta_\np\circ\Theta_t(x,\uy, \mathcal{G}(x,t,\uy))$, $\chi_\eq(x,t,y', \oy) :=\chi_1(x) \eta_{\eq}\circ\Theta_t\left(x,H_1(x,t,y'), y', \oy\right)$ and $\chi_\im (x,t,\uy):=\chi_1(x) \eta_{\im}\circ\Theta_t\left(x,\uy, \oH(x,t,y')\right)$. Observe that $x$ and $y'$ are small in the supports of all three functions, $y_1$ is small in the support of $\Tilde{\chi}^\np$ and $\oy$ is small in the support of $\chi_\eq$. Further, due to our assumptions on the supports of $\eta_\np, \eta_\eq, \eta_\im$, it follows that $\mathcal{G}(x,t,\uy)$, $H_1(x,t,y',\oy)$ and $\oH(x,t,\uy)$ are of size about one in the support of $\Tilde{\chi}_\np, \chi_\eq$ and $\chi_\im$ respectively.

We can thus write 
\begin{equation}
\label{eq convol def eq im}
f*\mu^{\eq}_t(x)=\int \chi_{\eq}(x,t,y', \oy) f\left(H_1(x,t,y',\oy), y',\oy\right)\,d y'd\oy,\,\, f*\mu^{\im}_t(x)=\int \chi_{\im}(x,t,\uy) f\left(\oH(x,t,\uy), \uy\right)\,d\uy    
\end{equation}
and 
\[f*\mu^{\np}_t(x)=\int \Tilde{\chi}_{\np}(x,t,\uy) f\left(\uy, \mathcal{G}(x,t,\uy)\right)\,d\uy.\]

When considering the last integral above (at the north pole), it will be convenient to introduce a shear transformation in the $x$-variables (depending smoothly on $t$)
\[\underline \fx(x,t)=\underline x, \,\,\,
\overline \fx(x,t)=\ox+t^2. \]
By this change of variables
\begin{equation}
\label{eq convol def np}
f*\mu^{\np}_t(\underline \fx, \overline \fx-t^2)=\int \chi_{\np}(\fx,t,\uy) f\left(\uy, G(\fx,t,\uy)\right)\,d\uy,    
\end{equation}
with
\begin{equation}
    \label{eq np form}
     G(x,t,\uy):=\overline{x} + \frac{1}{2}\underline{x}^T J \underline{y}+t^2\frac{|\ux-\uy|^4}{2t^4}+\frac{|\ux-\uy|^8}{t^8}R(\frac{\ux-\uy}{t})
\end{equation}
and $\chi^\np(x,t,y)=\Tilde{\chi}_\np(\ux,\ox+t^2, y)$. Note that while the shear transform leaves the Lebesgue space estimates unchanged, the smooth cut off $\chi^\np$ is now supported in the set where both $\ox$ and $G(x,t,\oy)$ are of size about one, while $\ux$ and $\uy$ are small.

The right hand sides of \eqref{eq convol def eq im} and \eqref{eq convol def np} represent operators with Schwartz kernel $K$ of three types depending on the neighborhood of localization:
\begin{align*}
    \textbf{Near the Equator:  } K(x,t,y)&=\chi_{\eq}(x,t,y')\delta_0\left(H_1(x,t,y',\oy)-y_1\right). \\
    \textbf{In the Intermediate Region:  } K(x,t,y)&=\chi_{\im}(x,t,\uy)\delta_0\left(\oH(x,t,\uy)-\oy\right).\\
     \textbf{Near the North Pole:  } K(x,t,y)&= \chi_{\np}(x,t,\uy)\delta_0\left(G(x,t,\uy)-\oy\right).
\end{align*}
By applying the Fourier transform to $\delta_0$, we can express
\begin{equation*}
    \label{eq: kernel integral}
    K(x,t,y)=\chi(x,t,\Tilde{y})\int_{\theta\in\bbR}e^{i\psi(x,t,y,\theta)}\,\frac{d\theta}{2\pi},
\end{equation*}
where the phase function around the north pole, the equator and in the intermediate region is given by $\psi^\np(x,t,y,\theta):=\theta(G(x,t,\uy)-\oy)$, $\psi^\eq(x,t,y,\theta):=\theta(H_1(x,t,y',\oy)-y_1)$ and $\psi^\im(x,t,y,\theta):=\theta(\oH(x,t,\uy)-\oy)$ respectively; the smooth cut-off $\chi$ can be  $\chi_\np, \chi_\eq$ or $\chi_\im$, and $\Tilde{y}$ can be $\uy$ or $(y', \oy)$ depending on the support of the kernel.
Then $K$ is well defined as an oscillatory integral distribution. 

Next, we perform a dyadic decomposition in the frequency variable $\theta$.
Let $\zeta_0$ be a smooth radial function on $\R$ with compact support in $\{|\theta|< 1\}$ such that $\zeta_0(\theta)=1$ for $|\theta|\le 1/2$. We set $\zeta_1(\theta)=\zeta_0(\theta/2)-\zeta_0(\theta)$ and $\zeta_k(\theta)= \zeta_1(2^{1-k}\theta)$. For
$k\ge 0$, $\psi$ of any of the three forms above and corresponding $\chi$, we define 
\[ A^k f(x, t)= A^k_{t} f(x) = \int_{\bbR^{2n+1}} 
\chi(x,t, \Tilde{y})\int_{\theta\in \bbR}\zeta_k(\theta) 
e^{i\psi(x,t,y,\theta)} \tfrac{d\theta}{(2\pi)}\, f(y) dy,  \]
where $\Tilde{y}$ can be $\uy$ or $(y',\oy)$ depending on the support of the kernel.
Further, let
\[ M^kf(x)=  \sup_{t\in [1,2]} |A^k_{t} f(x)|.\] 
Following are the main estimates for $M^k$, which hold uniformly for all three possible forms of the phase function $\psi$ in the support of the corresponding smooth cut off $\chi$.

\begin{proposition}\label{prop:Akmax} 
(i) For $1\le p\le \infty$, 
\Be\label{eqn:maxLpest}
\| M^k f\|_p \lc  2^{\frac{k}p} 2^{-k(2n)\min(\frac1p,\frac1{p'})}\|f\|_p.\Ee
(ii) For $2\le q\le \infty$, 
\Be \label{eqn:maxantidiagest}
\|M^k f\|_{q} \lesssim 2^{k (1-\frac{2n+1}{q})} \|f\|_{q'}.\Ee
(iii) Set $q_5:=\tfrac{2(n+2)}{n}$. 
For every $\epsilon>0$, there exists $C_\epsilon>0$ such that 
 \begin{equation} \label{eqn:maxL2qestnp}
  \| M^k f \|_{L^{q_5, \infty}} \leq C_\epsilon 2^{-k\left(\frac{n^2-1}{n+2}-\epsilon\right)}\|f\|_2.
 \end{equation}
\end{proposition} 
\begin{proof}[Proof of Theorem \ref{thm:max full}, given Proposition \ref{prop:Akmax}] \label{sec:interpolpf}It suffices to show the required bounds for $\cM f(x):= \sum_{k\ge 0} M^k f$.

We shall show that $\mathcal{M}$ is of restricted weak type $(p,q)$ for $(\tfrac{1}{p}, \tfrac{1}{q})\in\{Q_2, Q_3\}$. Observe that for $n\geq 2$, we have $1-\tfrac{2n+1}{2}<0$. To deduce the required restricted weak type estimates for $\mathcal{M}$ at $Q_2, Q_3,$
we recall the Bourgain  interpolation argument (\cite{Bourgain-CompteRendu1986}, \cite{CarberySeegerWaingerWright1999}):
Suppose we are given sublinear operators $T_k$ so that for $k\ge 1$,
\[ \|T_k\|_{L^{p_0,1}\to L^{q_0,\infty}} \lesssim 2^{k a_0}\;\text{and}\; \|T_k\|_{L^{p_1,1}\to L^{q_1,\infty}} \lesssim 2^{-k a_1}\]
for some $p_0, q_0, p_1, q_1\in [1,\infty], a_0, a_1>0$. Then the operator $\sum_{k\ge 1} T_k$ is of restricted weak type $(p,q)$, where 
\[ (\tfrac1p, \tfrac1q, 0) = (1-\vartheta) (\tfrac1{p_0},\tfrac1{q_0}, a_0) + \vartheta(\tfrac1{p_1},\tfrac1{q_1}, -a_1) \]
and $\vartheta=\tfrac{a_0}{a_0+a_1}\in (0,1)$. 

The restricted weak type estimate for $\mathcal{M}$ at $Q_2=(\tfrac{2n}{2n+1},\tfrac{2n}{2n+1})$ now follows from
\eqref{eqn:maxLpest}. 
Similarly, the restricted weak type bound at $Q_3=(\tfrac{2n}{2n+1},\tfrac{1}{2n+1})$ follows from \eqref{eqn:maxantidiagest}. Interpolation between the estimates at $Q_1, Q_2$ and $Q_3$ yields the $L^p\to L^q$ boundedness of $\mathcal{M}$ for $(\tfrac{1}{p}, \tfrac{1}{q})$ in the interior of $\Delta Q_1Q_2Q_3$, on the open boundary  segment $(Q_2,Q_3)$, and on the half open boundary segment $[Q_1,Q_2)$

At \[Q_4=(\tfrac1{p_4},\tfrac1{q_4})=\left(\tfrac{n(2n+1)}{2n^2+2n+1},\tfrac{n}{2n^2+2n+1}\right),\]
due to the loss of a $2^{k\epsilon}$ factor on the right hand side of $\eqref{eqn:maxL2qestnp}$, we only have
\[\left(\tfrac{1}{p_4}-\epsilon,\tfrac{1}{q_4}+\tfrac{n\epsilon}{n+2},-\epsilon'\right)=(\vartheta+2\epsilon)\left(\tfrac{1}{2},\tfrac{n}{2(n+2)}, b+\epsilon\right)+(1-\vartheta-2\epsilon)(1,0,1),\]
with with $b=-\tfrac{n^2-1}{n+2}$, $\vartheta=\tfrac{2(n+2)}{2n^2+2n+2}\in (0,1)$ and $\epsilon'=2\epsilon(1-\epsilon)-(2b+\vartheta)>0$ for $n\geq 2$. Letting $Q_{4, \epsilon}=\left(\tfrac{1}{p_\epsilon},\tfrac{1}{q_\epsilon}\right)=\left(\tfrac{1}{p_4}-\epsilon,\tfrac{1}{q_4}+\tfrac{n\epsilon}{n+2}\right)$, standard interpolation between \eqref{eqn:maxL2qestnp} and the case $q=\infty$ of \eqref{eqn:maxantidiagest} implies that
\[\|\cM f\|_{L^{q_\epsilon}}\lesssim_\epsilon \sum_{k\geq 0}2^{-k\epsilon'}\|f\|_{L^{p_\epsilon}}\lesssim_\epsilon \|f\|_{L^{p_\epsilon}}.\]
Interpolating again between the estimates at $Q_1, Q_2, Q_3$ and $Q_{4, \epsilon}$ yields the $L^p\to L^q$ boundedness of $\mathcal{M}$ for $(\tfrac{1}{p}, \tfrac{1}{q})$ in the interior of the quadrilateral formed by those vertices. Letting $\epsilon\to 0$, we conclude the same for $(\tfrac{1}{p}, \tfrac{1}{q})$ in the interior of $\mathcal{R}$.
Since bounds for $\mathcal{M}$  
imply bounds for $M$, 
this concludes the proof of part (i) of Theorem \ref{thm:max full}. 
\end{proof}
\subsection{Reduction to estimates on the averaging operator}
We consider a discrete dyadic cover of the interval $[1,2]$. 
Given a non-negative integer $k$, let $\cZ_k$ denote the set of left endpoints of all dyadic intervals of the form $(\nu2^{-k}, (\nu+1)2^{-k})$ (with $\nu\in \mathbb{Z}$) which intersect $[1,2]$, endowed with the counting measure. It is clear that  $\#(\cZ_k)= 2^k$. An application of the fundamental theorem of calculus gives the following pointwise bound
\begin{equation}
    \label{eq FTC}
    M^{k}f(x)\leq \sup_{t\in \cZ_{k}}|A^{k}_{t}f(x)|+\int_0^{2^{-k}}|\partial_sA^{k}_{t+s}f(x)|\,ds.
\end{equation}
Taking the $L^q$ norm on both sides, using the triangle inequality, then dominating the supremum by the $\ell^q(\cZ_k)$ norm and using the Minkowski inequality for the second term, we obtain
\[
\|M^{k} f\|_{L^{q}}\leq \| A^{k} f\|_{L^{q}(\R^{2n+1}\times \cZ_{k})}+\int_0^{2^{-k}}\|\partial_sA^{k}_{t+s}f(x)\|_{L^{q}(\R^{2n+1}\times \cZ_{k})}\,ds. \]

By a change of variables and using homegeneity (in $\theta$) for $k\geq 1$, we have
\begin{equation}\label{eq:Ak}  A^k_t f(x) = 2^k \int_{\bbR^{2n+1}}\int_{\bbR} e^{i2^k \psi(x,t,y,\theta)} b(x,t,\Tilde{y},\theta) f(y)\, d\theta\,dy 
\end{equation}
where 
the symbol $b$ is compactly supported in $\bbR^{2n+1}\times\bbR\times \bbR^{2n-1}\times \bbR$. It can be of the form $b_\np, b_\eq$ or $b_\im$, with  $\text{supp}(b_\np)\subseteq \text{supp}(\chi_\np)\times \text{supp}(\zeta_1)$ and $b_\eq, b_\im$ defined analogously ($\Tilde{y}$ might stand for $\uy$ or $(y',\oy)$ depending on the support). In particular, $|\theta|\in [1/2,2]$ in the support of all three types of symbols.

Using homegeneity, it is not hard to see that the operator $2^{-k} \frac{d}{dt} A^k_t$ is a linear combination of expressions involving operators which obey the same Lebesgue space estimates as the operator $A^k_t$. Thus parts (i) and (ii) of Proposition \ref{prop:Akmax} are a consequence of the following fixed time estimates.

 \begin{proposition} \label{prop:fix-time-est} Let $t\in [1,2]$. 
 (i) For $1\le p\le \infty$
 \Be \label{eqn:sptimeLpest}
\| {A}^k_tf \|_{p} \lesssim 2^{-k (2n)\min(\frac1p,\frac1{p'})}\|f\|_p, 
\Ee
with the implicit constant independent of $t$.

(ii) For $2\le q\le \infty$, 
\Be \label{eqn:sptimeantidiagest}
\|{A}^k_t f\|_{ q} \lesssim 2^{k (1-\frac{2n+2}{q})}\|f\|_{q'} ,\Ee
with the implicit constant independent of $t$.
\end{proposition}
The above fixed time estimates are sufficient to establish the positive result in Theorem \ref{thm single avg}.
\begin{proof}[Proof of Theorem \ref{thm single avg}, given Proposition \ref{prop:fix-time-est}]
Clearly, $\|f*\mu\|_{L^1}\lesssim \|f\|_{L^1}$ and $\|f*\mu\|_{L^\infty}\lesssim \|f\|_{L^\infty}$.
For $1<p<\infty$, \eqref{eqn:sptimeLpest} implies that
$\|\sum_{k\geq 0}A^k_t\|_{L^p\to L^p}\lesssim \sum_{k\geq 0}2^{-k(2n)\min(\tfrac{1}{p}, \tfrac{1}{p'})}\lesssim 1$. 

Further, since $1-\frac{2n+2}{q}=0$ for $q=\tfrac{1}{2n+2}$, \eqref{eqn:sptimeantidiagest} implies that $\|A^k_t\|_{L^{q'}\to L^q}\lesssim 1$ uniformly in $k$. 
To combine the $A^k_t$, we use standard applications of Littlewood-Paley theory, writing
$A^k_t = L_kA^k_t L_k +E_k$ where the $L_k $ satisfy Littlewood-Paley inequalities \[\Big\| \Big(\sum_{k\ge 0}|L_k f|^2\Big )^{1/2}\Big \|_r\lc \|f\|_r, \quad \Big\| \sum_{k\ge 0} L_k f_k\Big \|_r\lc \Big\|\Big(\sum_{k\ge 0} | f_k|^2\Big)^{1/2}\Big  \|_r\]  for $1<r<\infty$ and the error term $E_k$ has $L^p\to L^q$ operator norm $O(2^{-k})$ for all $1\le p,q\le \infty$. Since $q'\le 2\le q$, applying the Littlewood-Paley inequalities in conjunction with Minkowski's inequalities allows us to deduce the endpoint estimate at $(\tfrac{1}{q}, \tfrac{1}{q'})=(\frac{2n+1}{2n+2}, \frac{1}{2n+2})$. A further interpolation yields the positive result in the closure of the triangle formed by $(0,0)$, $(1,1)$ and $(\frac{2n+1}{2n+2}, \frac{1}{2n+2})$, thus finishing the proof.
\end{proof}
To prove the estimate at $q_5$, taking an $L^{q_5, \infty}$ norm on both sides of \eqref{eq FTC}, we conclude that
\[
\|M^{k} f\|_{L^{q_5, \infty}}\leq \| A^{k}_{t} f\|_{L^{q_5, \infty}(\R^{2n+1}\times \cZ_{k})}+\int_0^{2^{-k}}\|\partial_sA^{k}_{t+s}f(x)\|_{L^{q_5, \infty}(\R^{2n+1}\times \cZ_{k})}\,ds. 
\]
Thus part (iii) of Proposition \ref{prop:Akmax} is a consequence of the following estimate.
\begin{proposition}
\label{prop:disc q5}
Let $q_5=\frac{2(n+2)}{n}$.
We have
\Be \label{eqn:disc q5 np}
\| A^{k} \|_{L^2(\bbR^{2n+1})\to L^{q_5, \infty}(\R^{2n+1}\times \cZ_{k})} \lesssim
k2^{-k (\frac{n^2-1}{n+2}) }. 
\Ee
\end{proposition}

For all $k\ge 1$, 
we have the following estimates
\begin{gather}\label{eqn:trivestimate}
\|A_{t}^k\|_{L^1\to L^1} + \|A_{t}^k \|_{L^\infty\to L^\infty}  \lesssim 1,
\\ \label{eqn:trivL1infestimate}
\|A_{t}^k\|_{L^1\to L^\infty}  \lesssim 2^{k}.
\end{gather}
In view of the above it suffices in what follows to consider the case of large $k$. The bounds \eqref{eqn:sptimeLpest}, \eqref{eqn:sptimeantidiagest} then follow by an interpolation argument using 
\eqref{eqn:trivestimate}, \eqref{eqn:trivL1infestimate} and the fixed time $L^2$ estimate
\begin{equation}\label{eqn:sptimeL2estimate}
\| A_{t}^k f\|_{L^2(\R^{2n+1})} \lesssim 
2^{-k \frac{2n}{2} }
\|f\|_2.
\end{equation}
\subsection{Reduction to oscillatory integral estimates}
To prove \eqref {eqn:sptimeL2estimate} we use oscillatory integral operators of the form 
\begin{equation}
    \label{eq Tk def}
    T^k_t f(x)=T^k f(x,t) := \int_{\R^{2n+1}} e^{i 2^k \Phi(x,t,y)} b(x,t,y) f(y) dy,
\end{equation}
where $\Phi$ can be of the form $\Phi^\np(x,t,y)=\oy
G(x,t,\uy)$, $\Phi^\eq(x,t,y)= y_1H_1(x,t,y',\oy),$ and $\Phi^\im(x,t,y)=\oy \oH(x,t,\uy)$, 
with $G, H_1$ and $\oH$ as defined in equations \eqref{eq np form}, \eqref{eq equator form} and \eqref{eq intermediate form}. The corresponding amplitude $b$ is compactly supported in $\bbR^{2n+1}\times\bbR\times \bbR^{2n-1}\times \bbR$, and can be of the form $b^\np, b^\eq$ or $b^\im$ as defined in \eqref{eq:Ak} (with the role of $\theta$ played by $\oy$ or $y_1$, depending on the situation).

To reduce matters to oscillatory integral operators of the above form, we define \[F_k(y)= \int_{\bbR} f(y, \bar w) e^{-i2^k\bar y\bar w} d\bar w\] and
\[\Tilde{F}_k(y)= \int_{\bbR} f(w_1, y',\oy) e^{-i2^ky_{1}w_{1}} dw_{1}. \] 
The integrals on the right can be interpreted as scaled Fourier transforms in the $\oy$ and $y_1$ variables respectively. After renaming $\bar w$ to $\oy$ and $\oy$ to $\theta$, we have
\[ A^k f(x,t)= T^k F_k (x,t)\] 
around the north pole and in the intermediate regions. Similarly, renaming $w_1$ to $y_1$ and $y_1$ to $\theta$, we obtain
\[ A^k f(x,t)= T^k \Tilde{F}_k (x,t)\] 
around the equator.
By Plancherel's theorem $\|F_k\|_2=\|\Tilde{F}_k\|_2=(2^{-k}2\pi)\|f\|_2$.  Hence \eqref {eqn:sptimeL2estimate}
follows from 
\begin{proposition} \label{prop:L2osc-est} 
For all $f\in L^2(\bbR^{2n+1})$ and $t\in [1,2]$, 
\Be\label{fix time Tkest} 
\|T^k_t f\|_{ L^2(\bbR^{2n+1})} \lc 2^{-k\frac{2n+1}{2}} \|f\|_2,
\Ee
with the implicit constant uniform in $t\in [1,2]$. 
\end{proposition} 
The proof will be given using H\"ormander's standard  $L^2$ estimate \cite{Hormander1973} (also see \cite[Ch. IX.1]{Stein-harmonic}) combined with almost-orthogonality arguments. By the same argument, the $L^2\to L^{q_5, \infty}$ bound  
\eqref{eqn:disc q5 np} is reduced to the following estimate.
\begin{proposition}
\label{prop:Q5spacetimeTk}
Let  $q_5=\tfrac{2(n+2)}{n}$. For $f\in L^2(\bbR^{2n+1})$, we have
\Be \label{eqn:disc oio q5 np}
\| T^{k}f\|_{L^{q_5, \infty}(\R^{2n+1}\times \cZ_{k})} \lesssim
k2^{-k\frac{n(2n+1)}{2(n+2)}}\|f\|_{2}. 
\Ee
\end{proposition}

\section{A Scaling Argument at the Poles}
\label{sec scaling}
In this section, we focus on the case when the kernel of the operator 
$T^k$ is supported around the poles (more specifically, around the north pole but the same argument works for the south pole due to symmetry about the equator). In fact, the $L^p$ improving property at $Q_4$ is completely determined by how the operator behaves around the poles. We shall decompose the kernel of the operator $T^k$ based on the distance from the north pole and use a stretching argument to precisely understand this behaviour. 

Let $\rho_0$ be a smooth radial function on $\R^{2n}$ with compact support in $\{|\uy|< 1/4\}$ such that $\rho_0(\uy)=1$ for $|\uy|\le 1/8$. Setting $\rho_1(\uy)=\rho_0(\uy/2)-\rho_0(\uy)$, $\rho_{\ell}(\uy)= \rho_1(2^{\ell}\uy)$ for $1\leq\ell<\tfrac{k}{4}$ and $\rho_{\ell}(\uy)=\rho_0(2^{\ell}\uy)$ for $\ell=\frac{k}{4}$, we define
\[T_t^{k, 0} f(x): = \int_{\R^{2n+1}} e^{i 2^k \Phi(x,t,y) } b(x,t,y)\left(1-\rho_0\left(\frac{\ux-\uy}{2t}\right)\right) f(y) dy,\,\,\, \text{ for }  \ell=0,\]
\[T_t^{k,\ell} f(x) := \int_{\R^{2n+1}} e^{i 2^k \Phi(x,t,y) } b(x,t,y)\rho_{\ell}\left(\frac{\ux-\uy}{t}\right) f(y) dy,\,\,\, \text{ for } 1\leq \ell\leq \frac{k}{4}\]
and \[T^{k,\ell} f(x, t):=T_t^{k,\ell} f(x).\]
Here 
\[\Phi=\Phi^\np:=\oy G(x,t,\uy)
=\oy\left(\overline{x}+ \frac{1}{2}\underline{x}^T J \underline{y}+t^2\tfrac{|\ux-\uy|^4}{2t^4}+t^2\frac{|\ux-\uy|^8}{t^8}R(\frac{\ux-\uy}{t})\right),\] with $R$ as defined in \eqref{eq Gerror}, and $b=b^\np$ with  $\text{supp}(b_\np)\subseteq \text{supp}(\chi_\np)\times \text{supp}(\zeta_1)$. Then $T_t^{k} f=\sum_{\ell=0}^{\frac{k}{4}}T_t^{k,\ell} f$. In fact, since $b$ is supported in the set where $\frac{|\ux-\uy|}{t}<2^{-400n}$, the above decomposition is relevant for large values of $\ell$ (say $\ell\geq 400n$).
We make a change of variables $2^{\ell}\uy\to\uy$, which gives
\[T_t^{k,\ell} f(x) = 2^{-2n\ell}\int_{\R^d} e^{i 2^k \Phi(x,t,2^{-\ell}\uy,\oy) } b(x,t,2^{-\ell}\uy,\oy)\rho_{1}(2^{\ell}\ux-\uy) f(2^{-\ell}\uy,\oy) dy.\]
Observe that
\begin{align*}
\Phi(x,t,2^{-\ell}\uy,\oy)&=\overline{y}\left(\overline{x}+ \underline{x}^T J (2^{-\ell}\underline{y})+2^{-4\ell}t^2\tfrac{|2^{\ell}\underline{x}-\underline{y}|^4}{t^4}+2^{-8\ell}t^2\tfrac{|2^{\ell}\underline{x}-\underline{y}|^8}{t^8}R(\tfrac{\underline{x}-2^{-\ell}\underline{y}}{t}) \right)\\
&=2^{-4\ell}\overline{y}\left(2^{4\ell}\overline{x}+ \frac{1}{2}\underline{x}^T 2^{3\ell}J\underline{y}+t^2\tfrac{|2^{\ell}\underline{x}-\underline{y}|^4}{t^4}+2^{-4\ell}t^2\tfrac{|2^{\ell}\underline{x}-\underline{y}|^8}{t^8}R(\tfrac{2^{-\ell}(2^{\ell}\underline{x}-\underline{y})}{t}) \right).
\end{align*}
We define
\begin{equation}
    \label{eq gldef1}
    g(\uw):=\frac{|\uw|^4}{2}\,\text{ and } \Tilde{g}_{\ell}(\uw):=|\uw|^8R(2^{-\ell}\uw),
\end{equation}
so that
\begin{equation}
    \label{eq gldef2}
    g_{\ell}(\uw):=g(\uw)+2^{-4\ell}\Tilde{g}_{\ell}(\uw).
\end{equation}
Due to condition \eqref{eq Gerror} on $R$, we have 
\begin{equation}
    \label{eq ggerror}
     \left|\frac{\partial^\beta}{\partial \uw^{\beta}}\Tilde{g}_\ell(\uw)\right|\lesssim 100n\cdot 2^{-4\ell}|\uw|^{8-|\beta|}\,\, \text{ for } |\beta|\in\{0,1,2,3\} \text{ and } |\uw|\leq\frac{1}{2}.
\end{equation}
In fact, $g_{\ell}$ can be expressed as $g_{\ell}(\uw)=u_{\ell}(|\uw|)$, where $u$ is a smooth function in one variable satisfying the condition \begin{equation}
    \label{eq homerror}
    \left|\left(\frac{d}{dr}\right)^j\left(u_{\ell}(r)-\frac{r^4}{2}\right)\right|\lesssim 100n\cdot2^{-4\ell}r^{8-j}\,\, \text{ for } j\in\{0,1,2,3\} \text{ and } |r|\leq\frac{1}{2}.
\end{equation}
We thus have
\[T_t^{k,\ell} f(2^{-3\ell}\ux,2^{-4\ell}\ox) = 2^{-2n\ell}\int_{\R^{2n+1}} e^{i 2^{k-4\ell} \Phi_{\ell}(x,t,y)} a_{\ell}(x,t,y) f(2^{-\ell}\uy,\oy) dy,\]
where
\[\Phi_{\ell}(x,t,2\uy,\oy)=\overline{y}\left(\overline{x}+ \frac{1}{2}\underline{x}^T J \underline{y}+t^2g_{\ell}(\tfrac{2^{-2\ell}\underline{x}-\underline{y}}{t}) \right),\]
and for a fixed $t\in[1,2]$, $a_{\ell}(x,t,y)=b(2^{-3\ell}\ux,2^{-4\ell}\ox,t,2^{-\ell}\uy,\oy)\rho_{1}(2^{-2\ell}\ux-\uy)$ is supported in the set where $|\ux|\lesssim 2^{3\ell}, |\ox|\lesssim 2^{4\ell}$, $\frac{1}{8}\leq|2^{-2\ell}\ux-\uy|\leq \frac{1}{2}$ (resp. $|2^{-2\ell}\ux-\uy|\leq \frac{1}{4}$) for $\ell<\frac{k}{4}$ (resp. $\ell=\frac{k}{4}$) and $|\oy|\sim 1$.
\begin{rem}
\label{rem diff scal}
Observe that the input variable $\uy$ is scaled by $2^{-\ell}$, while the output variable $\ux$ is scaled by a factor of $2^{-3\ell}$. This is due to the presence of the bilinear term involving the matrix $J$ in the phase function. This phenomenon was not present in \cite{iosevich1994maximal}, where a scaling type type argument was used to establish the $L^p$ boundedness of global maximal operators associated to families of flat, finite type curves in $\bbR^2$ (also see \cites{manna2017pl}).
\end{rem}
We introduce an operator $\cT_t^{k,\ell}$ defined by
\begin{equation}
     \label{eq cTkl def}
     \cT_t^{k,\ell} f(\ux,\ox) = \int_{\R^d} e^{i 2^{k-4\ell} \Phi_{\ell}(x,t,y)} a_{\ell}(x,t,y) f(\uy,\oy) dy
\end{equation}
and let 
\begin{equation}
\cT^{k,\ell} f(\ux,\ox, t):=\cT_t^{k,\ell} f(\ux,\ox).    
\end{equation}
 
Further, for $a,b\in \bbZ$, let $\tau_{a,b}$ be the dilation defined by
$
\tau_{a,b}(\ux,\ox)=(2^{-a}\ux,2^{-b}\ox),   
$
so that
\[T^{k,\ell}_tf=2^{-2n\ell} \cT^{k,\ell}_t(f\circ\tau_{\ell,0})\circ\tau_{-3\ell,-4\ell}.\]
The following lemma will be useful for $TT^*$ type arguments in the upcoming sections.
\begin{lemma}
\label{lem TT* scaling}
Let $a,b,c\in \bbZ$ and for $i=1,2$, let $\cU_i,U_i$ be operators defined on $\bbR^{2n+1}$ such that
\begin{equation}
    \label{eq TT* scaling}
    U_i f=2^{-2nc}\cU_i(f\circ\tau_{c,0})\circ\tau_{-a,-b}.
\end{equation}
Then for $1\leq p,q\leq \infty$,
\[\|U_1U_2^*\|_{L^p\rightarrow L^q}\leq 2^{-2n(a+c)}2^{-b}2^{(2na+b)(\frac{1}{p}-\frac{1}{q})}\|\cU_1\cU_2^*\|_{L^p\rightarrow L^q}.\]
\end{lemma}
\begin{proof}
As a consequence of \eqref{eq TT* scaling}, we have \[U_2^*f= 2^{-(2na+b)}\cU_2^*(f\circ\tau_{a,b})\circ\tau_{-c,0}.
\]
It follows that \begin{align*}
U_1U_2^*f=2^{-2nc}\cU_1[(U_2^*f)\circ\tau_{c,0}]\circ\tau_{-a,-b}
&=2^{-2n(a+c)}2^{-b}\cU_1[\cU_2^*(f\circ\tau_{a,b})\circ\tau_{-c,0}\circ\tau_{c,0}]\circ\tau_{-a,-b}\\
&=2^{-2n(a+c)}2^{-b}\cU_1\cU_2^*(f\circ\tau_{a,b})\circ\tau_{-a,-b}.
\end{align*}
Taking $L^q$ norms on both sides and changing variables on the right hand side, we get
\[\|U_1U_2^*f\|_{L^q}=2^{-2n(a+c)}2^{-b}2^{-\frac{2na+b}{q}}\|\cU_1\cU_2^*(f\circ\tau_{a,b})\|_{L^q}\leq 2^{-2n(a+c)}2^{-b}2^{-\frac{2na+b}{q}}\|\cU_1\cU_2^*\|_{L^p\to L^q}\|f\circ\tau_{a,b}\|_{L^p}. \]
Since $\|f\circ\tau_{a,b}\|_{L^p}=2^{\frac{2na+b}{p}}\|f\|_{L^p}$, this yields the desired estimate.
\end{proof}
To prove \eqref{eqn:disc oio q5 np}, we shall also need a fixed time $L^2$ estimate for each scaled piece $\cT^{k,\ell}_t$.
\begin{proposition}
\label{prop:L2cTkl-est} Let $\cT^{k,\ell}_t$ be as defined in \eqref{eq cTkl def}. For all $f\in L^2(\bbR^{2n+1})$ and $t\in [1,2]$, 
\begin{equation*}
\label{fix time cTklest} 
\|\cT^{k,\ell}_t f\|_{ L^2(\bbR^{2n+1})} \lc 2^{-(k-4\ell)\frac{2n+1}{2}} \|f\|_2,
\end{equation*}
with the implicit constant uniform in $t\in [1,2]$. 
\end{proposition}
The above proposition shall be proven in \S \ref{subsec L2 np}.
Combining it with Lemma \ref{lem TT* scaling} for $a=3\ell, b=4\ell$ and $c=\ell$ yields the following.
\begin{corollary}
\label{cor Tkl TT* L2}
For $t,t'\in [1,2]$,
\[\|T^{k,\ell}_t(T^{k,\ell}_{t'})^*\|_{L^2(\bbR^{2n+1})\rightarrow L^2(\bbR^{2n+1})}\lesssim 2^{-2n(3\ell+\ell)}2^{-4\ell}2^{-(k-4\ell)(2n+1)}=2^{-k(2n+1)},\]

\end{corollary}

\begin{rem}
The above corollary implies that $\|T^{k,\ell}_t\|_{L^2(\bbR^{2n+1})\rightarrow L^2(\bbR^{2n+1})}\lesssim 2^{-k(\frac{2n+1}{2})}$.
If we sum this up for $0\leq \ell\leq \frac{k}{4}$, we would obtain the fixed time $L^2$ estimate \eqref{fix time Tkest} for $T^k$ around the poles but up to a logarithmic loss in $2^k$. However, as we shall see in \S\ref{subsec L2 np}, the rotational curvature remains non-vanishing at the poles and thus \eqref{fix time Tkest} can be directly obtained without incurring any logarithmic loss for the estimates at $Q_1, Q_2$ and $Q_3$.
\end{rem}

\section{A Stein-Tomas Argument}
\label{sec stein tomas}
In this section, we shall reduce the estimate 
\eqref{eqn:disc oio q5 np}, for the operators $T^k$ (around the equator and in the intermediate region) and $T^{k,\ell}$ (around the poles), to an $L^{q_5', 1}\to L^{q_5, \infty}$ estimate for the respective operator precomposed with its adjoint via the $TT^*$ technique. Using a Stein-Tomas type argument, the said estimate will then be obtained by interpolating between an $L^2\to L^2$ and an $L^1\to L^\infty$ estimate. We first give the details when the kernel of $T^k$ is supported around the poles (more precisely, the north pole), where the argument needs to be run separately for each piece $T^{k,\ell}$. 

\subsection{Around the North Pole}
Since $T^k f=\sum_{\ell=0}^{\frac{k}{4}}T^{k,\ell} f$, to prove \eqref{eqn:disc oio q5 np}, it suffices to show 
\begin{equation*}
\| T^{k, \ell}\|_{L^2(\bbR^{2n+1})\to L^{q_5, \infty}(\R^{2n+1}\times \cZ_{k})} \lesssim
2^{-k\frac{n(2n+1)}{2(n+2)}}  
\end{equation*}
for each $0\leq \ell\leq \frac{k}{4}$. Using a $TT^*$ argument (and the fact that the dual space of $L^{q', 1}$ is $L^{q, \infty}$), the above estimate is a consequence of 
\begin{equation}
    \label{eq St-To l piece}
    \| T^{k, \ell} (T^{k, \ell})^*\|_{L^{q_5', 1}(\R^{2n+1}\times \cZ_{k})\to L^{q_5, \infty}(\R^{2n+1}\times \cZ_{k})} \lesssim
2^{-k\frac{n(2n+1)}{(n+2)}}.
\end{equation}
We define the operator $S^{k,\ell}$ acting on functions $g:\R^{2n+1}\times \mathcal{Z}_{k}\to \bbC$ by
\begin{equation}
    \label{eq skl}
    S^{k,\ell}g(x,t) := \sum_{t'\in\mathcal{Z}_{k}} T^{k,\ell}_{t} (T^{k,\ell}_{t'})^* [g(\cdot, t')](x).
\end{equation}
Then \eqref{eq St-To l piece} follows from 
\begin{equation}
    \label{eq St-To Sl piece}
    \| S^{k, \ell}g\|_{L^{q_5, \infty}(\R^{2n+1}\times \cZ_{k})} \lesssim 2^{-k\frac{n(2n+1)}{(n+2)}}\|g\|_{L^{q_5', 1}(\R^{2n+1}\times \cZ_{k})}.
\end{equation}    
For $m\ge 0$ and $t_{\nu}\in\mathcal{Z}_{k}$, let
\[ \mathcal{Z}_{k}^m(t) = \{ t'\in \mathcal{Z}_{k}\,:\,2^{-k+m}\le |t-t'|\le 2^{-k+m+1} \}. \]
Note that $\mathcal{Z}_{k}^m(\nu)$ is empty, if $m>k+4$. Further, we define
\begin{equation}
    \label{eq sklm}
    S^{k,\ell}_m g(x,t) := \sum_{t'\in\mathcal{Z}_{k}^m(t)} T^{k,\ell}_t (T^{k,\ell}_{t'})^* [g(\cdot, t')](x). 
\end{equation}
We observe 
\[\left(\tfrac{1}{q_5^\prime},\tfrac{1}{q_5},-\tfrac{n(2n+1)}{n+2},0\right)=\vartheta\left(\tfrac{1}{2},\tfrac{1}{2},-(2n+1),1\right)+(1-\vartheta)\left(1,0,0,-\tfrac{n}{2}\right),\]
with $\tfrac{\vartheta}{2}=\tfrac{1}{q_5}=\tfrac{n}{2(n+2)}$. Thus, using the Bourgain interpolation trick for $S^{k,\ell}=\sum_{m\geq 0}S^{k,\ell}_m$,  \eqref{eq St-To Sl piece} is a consequence of 
\begin{equation}
     \label{eq skl L2}
     \| S^{k, \ell}_m g\|_{L^{2}(\R^{2n+1}\times \cZ_{k})} \lesssim 2^{-k(2n+1)}2^m\|g\|_{L^2(\R^{2n+1}\times \cZ_{k})}
\end{equation}
and
\begin{equation}
\label{eq skl L1-Linf}
\| S^{k, \ell}_m g\|_{L^{\infty}(\R^{2n+1}\times \cZ_{k})} \lesssim 2^{-\frac{nm}{2}}\|g\|_{L^1(\R^{2n+1}\times \cZ_{k})}.    
\end{equation}
To show \eqref{eq skl L2}, using the Cauchy-Schwarz inequality, we have
\begin{align*}
\|S^{k,\ell}_mg\|_{L^2(\mathbb{R}^{2n+1}\times\cZ_{k})}
&=\left(\sum_{t\in \cZ_{k}}\int|\sum_{t'\in\mathcal{Z}_{k}^m(t)} T^{k,\ell}_{t} (T^{k, \ell}_{t'})^* [g(\cdot, t')](x)|^2\,dx\right)^{1/2}\\
 &\leq \left(\sum_{t\in \cZ_{k}}\#(\mathcal{Z}_{k}^m(t))\int\sum_{t'\in\mathcal{Z}_{k}^m(t)} |T^{k,\ell}_{t} (T^{k,\ell}_{t'})^* [g(\cdot, t')](x)|^2\,dx\right)^{1/2}\\
 &\lesssim 2^{-k(2n+1)}2^{m}\|g\|_{L^2(\mathbb{R}^{2n+1}\times\cZ_{k})},
\end{align*}
where we have used Corollary \ref{cor Tkl TT* L2} and the fact that
$\#(\mathcal{Z}_{k}^m(t))\sim 2^{m}$ for all $t\in \cZ_{k}$.

To prove \eqref{eq skl L1-Linf}, we need the following estimate on the kernel of the operator $\cT^{k,\ell}_t(\cT^{k,\ell}_{t'})^*$.
\begin{proposition}
\label{prop kernel est np}
Let $\cK^{k,\ell}_{t,t'}$ denote the kernel of $\cT^{k,\ell}_t(\cT^{k,\ell}_t)^*$, where $\cT^{k,\ell}_t$ and $\cT^{k,\ell}_{t'}$ are as defined in \eqref{eq cTkl def} for $t,t'\in [1,2]$. Then for $0\leq \ell<\frac{k}{4}$,

\begin{equation}
    \label{eq cKkl L1Linf}
    \|\cK^{k,\ell}_{t,t'}\|_{L^\infty(\bbR^{2n+1})}\lesssim (1+2^{k-4\ell}|t-t'|)^{-n}.
\end{equation}
\end{proposition}
Postponing the proof of the above proposition to \S \ref{subsec ker est np}, we use it for now to prove \eqref{eq skl L1-Linf}.
The above estimate implies that $\|\cT^{k,\ell}_t(\cT^{k,\ell}_{t'})^*\|_{L^1(\bbR^{2n+1})\rightarrow L^\infty(\bbR^{2n+1})}\lesssim (1+2^{k-4\ell}|t-t'|)^{-n}$. Using Lemma \ref{lem TT* scaling} with $a=3\ell, b=4\ell, c=\ell$, we obtain
\begin{align*}
\|T^{k,\ell}_t(T^{k,\ell}_{t'})^*\|_{L^1(\bbR^{2n+1})\rightarrow L^{\infty}(\bbR^{2n+1})}
&\lesssim 2^{-2n(3\ell+\ell)}2^{-4\ell}2^{(6n+4)\ell}(1+2^{k-4\ell}|t-t'|)^{-n}\\
&= 2^{-2n\ell}(1+2^{k-4\ell}|t-t'|)^{-n}.
\end{align*}

For $m\geq 4\ell$, using the estimate above, we have
\[\|S^{k,\ell}_m g\|_{L^1(\bbR^{2n+1}\times\cZ_k)\rightarrow L^{\infty}(\bbR^{2n+1}\times\cZ_k)}
\lesssim 2^{-2n\ell}(2^{k-4\ell}\times 2^{-k+m})^{-n}=2^{2n\ell}2^{-nm}\leq 2^{-\frac{nm}{2}}.\]
On the other hand, for $m\leq 4\ell$ or $\ell=\frac{k}{4}$, we can use the trivial bound of $2^{-2n\ell}$ (using the fact that $|\mathcal{K}^{k, \ell}_{t,t'}|\lesssim 1$ for $\ell=\frac{k}{4}$) to conclude that
\[\|S^{k,\ell}_m\|_{L^1(\bbR^{2n+1}\times\cZ_k)\rightarrow L^{\infty}(\bbR^{2n+1}\times\cZ_k)}
\lesssim 2^{-2n\ell}\leq 2^{-\frac{nm}{2}}.\]
This implies \eqref{eq skl L1-Linf}, and concludes the proof of \eqref{eqn:disc q5 np}.
\subsection{Around the Equator and in the Intermediate Region}
The proof of estimate \eqref{eqn:disc oio q5 np} when the operator $T^k_t$ has kernel supported away from the north pole proceeds in the exact manner as above. The following replaces Proposition \ref{prop kernel est np} in this setting.
\begin{proposition}
\label{prop kernel est eq}
Let $K^{k}_{t,t'}$ denote the kernel of $T^{k}_t(T^{k}_t)^*$, where $T^{k}_t$ and $T^{k}_{t'}$ are as defined in \eqref{eq Tk def} for $t,t'\in \cZ_k$, with $\Phi=\Phi^\eq$ or $\Phi=\Phi^\im$. Then  
    \begin{equation}
    \label{eq Kk L1Linf}
    \|K^{k}_{t,t'}\|_{L^\infty(\bbR^{2n+1})}\lesssim (1+2^{k}|t-t'|)^{-n}.
\end{equation}
\end{proposition}
The above Proposition shall be proved in \S\ref{subsubsec kernel eq} and \S\ref{subsubsec IM kernel}. Below we show how it implies \eqref{eqn:disc oio q5 np}. 
Using the $TT^*$ technique, it suffices to prove that \begin{equation*}
    \| T^{k} (T^{k})^*\|_{L^{q_5', 1}(\R^{2n+1}\times \cZ_{k})\to L^{q_5, \infty}(\R^{2n+1}\times \cZ_{k})} \lesssim
2^{-k\frac{n(2n+1)}{(n+2)}}.
\end{equation*}
We define $S^{k}$ and $S^k_m$ as in \eqref{eq skl} and \eqref{eq sklm}, with $T^{k,\ell}_t, T^{k,\ell}_{t'}$ replaced with $T^k_t, T^k_{t'}$ respectively. The above estimate then follows from
\begin{equation*}
    \| S^{k}g\|_{L^{q_5, \infty}(\R^{2n+1}\times \cZ_{k})} \lesssim 2^{-k\frac{n(2n+1)}{(n+2)}}\|g\|_{L^{q_5', 1}(\R^{2n+1}\times \cZ_{k})},
\end{equation*}  
which by Bourgain's interpolation trick is a consequence of the estimates
\begin{equation}
     \label{eq sk L2}
     \| S^{k}_m g\|_{L^{2}(\R^{2n+1}\times \cZ_{k})} \lesssim 2^{-k(2n+1)}2^m\|g\|_{L^2(\R^{2n+1}\times \cZ_{k})}
\end{equation}
and
\begin{equation}
\label{eq sk L1-Linf}
\| S^{k}_m g\|_{L^{\infty}(\R^{2n+1}\times \cZ_{k})} \lesssim 2^{-\frac{nm}{2}}\|g\|_{L^1(\R^{2n+1}\times \cZ_{k})}.    
\end{equation}
Estimate \eqref{eq sk L2} follows like in the previous case, using the Cauchy-Schwarz inequality and the fixed time estimate \eqref{fix time Tkest} for $T^k$ with kernel supported around the equator or in an intermediate region.

\eqref{eq sk L1-Linf} is a straightforward consequence of \eqref{eq Kk L1Linf}, for we have
\[\|S^{k,\ell}_m g\|_{L^1(\bbR^{2n+1}\times\cZ_k)\rightarrow L^{\infty}(\bbR^{2n+1}\times\cZ_k)}
\lesssim (2^{k}\times 2^{-k+m})^{-n}=2^{-nm}\leq 2^{-\frac{nm}{2}}.\]
This concludes the proof of \eqref{eqn:disc oio q5 np}.

Thus, in order to prove the positive results in Theorems \ref{thm single avg} and \ref{thm:max full}, we need to establish Propositions \ref{prop:L2cTkl-est} and \ref{prop kernel est np} around the poles (proven in \S\ref{subsec L2 cTkl np} and \S\ref{subsec ker est np} respectively), Proposition \ref{prop kernel est eq} around the equator and in the intermediate region (contained in \S\ref{subsubsec kernel eq} and \S\ref{subsubsec IM kernel} respectively) and Proposition \ref{prop:L2osc-est} in all three cases (proven in \S\ref{subsubsec L2 eq}, \S\ref{subsubsec IM L2} and \S\ref{subsec L2 np}).

\section{Results about Radial Functions and Matrix Inverses}
\label{sec matrix inv}
This section contains a few results about the Hessian of the radial function $g$ and its interaction with a skew symmetric matrix $J$ satisfying $J^2=-I_d$, which will aid us in the calculations to follow.

\begin{lemma}
\label{lem:radprop}
Let $d\in \bbN$, $u:\bbR\to \bbR$ be a smooth function and let $g:\bbR^d\to \bbR$ be given by $g(w)=u(|w|)$. Then   
    
    (i) $\nabla g(w)=\frac{u'(|w|)}{|w|}w$.
    
    (ii) $\det g''(w)=\left(\frac{u'(|w|)}{|w|}\right)^{d-1}u''(|w|)$.

\end{lemma}

\begin{proof}
The first part is a simple application of the chain rule, so we focus on the proof of the second part. Setting $r=|w|$ and taking partial derivatives twice, we have
\[\partial^2_{ij}g(w)=u''(r)\frac{w_iw_j}{r^2}+u'(r)\left(\delta_{ij}\frac{1}{r}-\frac{w_iw_j}{r^3}\right)\]
for $1\leq i,j\leq d$. Thus
\begin{equation}
    \label{eq:Hessform}
    g''(w)=\frac{u'(r)}{r}I_d+\left(\frac{u''(r)}{r^2}-\frac{u'(r)}{r^3}\right)ww^\intercal.
\end{equation}
Here $I_d$ denotes the identity matrix of order $d$.
Taking determinants on both sides, we have
\[\det g''(w)=\left(\frac{u'(r)}{r}\right)^d \det\left[I_d+\left(\frac{u''(r)}{u'(r)r}-\frac{1}{r^2}\right)ww^\intercal\right].\]
We claim that for any $\sigma\in \bbR$, 
\begin{equation}
    \label{eq:claimdet}
    \det(I_d+\sigma ww^{\intercal})=1+\sigma |w|^2=1+\sigma r^2.
\end{equation}
The claim then implies the desired equality, for then it follows that
\[\det g''(w)=\left(\frac{u'(r)}{r}\right)^d\left[1+\left(\frac{u''(r)}{u'(r)r}-\frac{1}{r^2}\right)r^2\right]=\left(\frac{u'(|w|)}{|w|}\right)^{d-1}u''(|w|).\]

We now prove \eqref{eq:claimdet} by induction on the dimension $d$. It can be easily verified for $d=1,2$. Suppose now that the claim is true for $d-1$ with $d\geq 3$. Expressing $w\in \bbR^d$ as $w=(w',w_d)$ with $w'\in \bbR^{d-1}$, we set $r'=|w'|$. Using Schur's complement, we have
\begin{align*}
    \det (I_d+\sigma ww^{\intercal})
    &=\det \begin{pmatrix}
    1+\sigma w_1^2 & \sigma w_1w_2 &\ldots & \sigma w_1w_d\\
    w_1w_2 & 1+\sigma w_2^2 & \ldots & \sigma w_2w_d\\
    \vdots &\vdots &\ddots &\vdots\\
    \sigma w_1w_d &\ldots & \ldots & 1+\sigma w_d^2
    \end{pmatrix}\\
    &=(1+\sigma w_d^2)\det \left[ I_{d-1}+\sigma w'w'^{\intercal}-\frac{\sigma^2w_d^2}{(1+\sigma w_d^2)}  w'w'^{\intercal}\right]\\
    &=(1+\sigma w_d^2)\det \left[ I_{d-1}+\frac{\sigma}{(1+\sigma w_d^2)}w'w'^{\intercal}\right].
\end{align*}
Using the induction hypothesis, we conclude that
\[\det (I_d+\sigma ww^{\intercal})=(1+\sigma w_d^2)\left(1+\frac{\sigma}{(1+\sigma w_d^2)}|w'|^2\right)=1+\sigma(|w'|^2+|w_d|^2)=1+\sigma|w|^2.\]
\end{proof}

Next we have a formula for the inverses of matrices of the form $J+g''(w)$ where $J$ is a skew-symmetric matrix of dimension $d$ with $J^2=-I_{d}$ and $w\in \bbR^{d}$. 
\begin{lemma}
\label{lem matrix inv}
Let $J$ be a skew-symmetric matrix with $J^2=-I_{d}$ and let $w\in\bbR^{d}$. For any $\sigma,\lambda, \kappa, \gamma\in \bbR$, we have
\begin{multline}
    \label{eq:invform}
    (\sigma I_{2n}+\lambda ww^{\intercal}+\kappa J+\gamma w(Jw)^{\intercal})^{-1}\\=\frac{1}{\sigma^2+\kappa^2}\bigg[\sigma I_{2n}-\kappa J+\frac{\gamma\kappa-\sigma\lambda}{\sigma^2+\kappa^2+(\sigma\lambda-\gamma\kappa)|w|^2}\left(\sigma ww^{\intercal}-\kappa(Jw)w^{\intercal}\right)\\
    +\frac{\gamma\sigma+\lambda\kappa}{\sigma^2+\kappa^2+(\sigma\lambda-\gamma\kappa)|w|^2}\left(-\sigma w(Jw)^{\intercal}+\kappa(Jw)(Jw)^{\intercal}\right)\bigg].
\end{multline}
\end{lemma}
\begin{proof}
This can be verified directly, using the facts that $J^2=-I_{d}$, $ww^\intercal w w^\intercal=|w|^2ww^\intercal$ and $w^\intercal J w=0$ (due to $J$ being skew-symmetric).
\end{proof}
The following corollaries shall be useful in later calculations.
\begin{corollary}
\label{cor matrix inv est}
Suppose there exist positive constants $c, C$ such that $\sigma^2+\kappa^2>c$ and $\max\{|\alpha|, |\lambda|, |\kappa|, |\gamma|, |w|\}<C$. Then it is clear from \eqref{eq:invform} that 
\begin{equation*}
\label{eq inv matrix est}
    \det\left[(\sigma I_{2n}+\lambda ww^{\intercal}+\kappa J+\gamma w(Jw)^{\intercal})^{-1}\right]\lesssim_{c, C, d} 1.
\end{equation*}
\end{corollary}
\begin{corollary}
\label{cor inn prod}
Observe that $(Jw)^{\intercal}w=0$ as $J$ is skew-symmetric. Thus, using \eqref{eq:invform}, we conclude that
\begin{align*}
&(\sigma I_{2n}+\lambda ww^{\intercal}+\kappa J+\gamma w(Jw)^{\intercal})^{-1}w\\
&=\frac{1}{\sigma^2+\kappa^2}\bigg[(\sigma I_{2n}-\kappa J)w+\frac{\gamma\kappa-\sigma\lambda}{\sigma^2+(\sigma\lambda-\gamma\kappa)|w|^2+\kappa^2} (\sigma ww^{\intercal}- \kappa(Jw)w^{\intercal})w \bigg]\\
&=\frac{1}{\sigma^2+(\sigma\lambda-\gamma\kappa)|w|^2+\kappa^2}(\sigma I_{2n}-\kappa J)w.
\end{align*}
\end{corollary}

\section{Proof of Propositions \ref{prop:L2cTkl-est} and \ref{prop kernel est np}: Estimates at the Poles}
\label{sec np}
We begin by proving Proposition \ref{prop:L2cTkl-est} about the fixed time $L^2$ bound on operators $\cT^{k,\ell}_t$ for $400n\leq \ell\leq \frac{k}{4}$.

\subsection{Proof of Proposition \ref{prop:L2cTkl-est}}
\label{subsec L2 cTkl np}
Recall that for $0\leq \ell\leq \frac{k}{4}$, the operator $\cT^{k,\ell}_t$ is defined as follows
\[
    \cT_t^{k,\ell} f(\ux,\ox) = \int_{\R^{2n+1}} e^{i 2^{k-4\ell} \Phi_{\ell}(x,t,y)} a_{\ell}(x,t,y) f(\uy,\oy) dy.
\]
Here
\begin{equation}
    \label{eq phase np Taylor}
    \Phi_{\ell}(x,t,\uy,\oy)=\overline{y}\left(\overline{x}+ \frac{1}{2}\underline{x}^T J \underline{y}+t^2g_{\ell}(\tfrac{2^{-2\ell}\underline{x}-\underline{y}}{t}) \right)
\end{equation}
and for a fixed $t\in[1,2]$, $a_{\ell}(x,t,y)=b(2^{-3\ell}\ux,2^{-4\ell}\ox,t,2^{-\ell}\uy,\oy)\rho_{1}(2^{-2\ell}\ux-\uy)$ is supported in the set where $|\ux|\lesssim 2^{3\ell}, |\ox|\lesssim 2^{4\ell}$, $\frac{1}{8}\leq|2^{-2\ell}\ux-\uy|\leq \frac{1}{2}$ (resp. $|2^{-2\ell}\ux-\uy|\leq \frac{1}{4}$) for $\ell<\frac{k}{4}$ (resp. $\ell=\frac{k}{4}$), and $|\oy|\sim 1$.

We seek to apply H\"ormander's classical $L^2$ bound
 (\cite[ch. IX.1]{Stein-harmonic}). For a \textit{fixed} phase function $\Phi$ supported in $\bbR^{2n+1}\times\bbR^{2n+1}$, this bound guarantees the fixed time estimate in Proposition \ref{prop:L2cTkl-est}, provided that the rank of the $(2n+1)\times (2n+1)$ mixed Hessian matrix $\Phi_{(x,y)}''$ is equal to $2n+1$. The implicit constant in the inequality depends on the lower bound for the determinant of $\Phi_{(x,y)}''$, the upper bound on finitely many derivatives of $\Phi$ and the amplitude $a$, and the size of the support of $a$.
 
In our case, we require uniform $L^2$ bounds for a family of operators with phase functions $\Phi_{\ell}$ and amplitudes $a_{\ell}$ for $400n\leq \ell\leq \frac{k}{4}$. From the definition of $a_\ell$, it is clear that for large $\ell$, the derivatives of $a_\ell$ upto a finite order are bounded from above by an absolute constant times the derivatives of $a$. Coming to the phase functions, recall that $g_\ell$ (in the definition of $\Phi_\ell$ above) is given by
\begin{equation*}
g_{\ell}(\uw)=g(\uw)+2^{-4\ell}\tilde{g}_\ell(\uw),
\end{equation*} with $\Tilde{g}_\ell$ satisfying property \eqref{eq ggerror}, namely
\begin{equation*}
     \left|\frac{\partial^\beta}{\partial \uw^{\beta}}\Tilde{g}_\ell(\uw)\right|\leq 100n\cdot 2^{-4\ell}|\uw|^{8-|\beta|}\,\, \text{ for } |\beta|\in\{0,1,2,3\} \text{ and } |\uw|\leq\frac{1}{2}.
\end{equation*}
The above estimate is also true for higher order derivatives, perhaps with an absolute constant greater than $100n$. Thus for large enough $\ell$ (say $\ell \geq 400n$),  $g_{\ell}$ and its higher order derivatives essentially behave like $g$ and consequently $\Phi_{\ell}$ is a smooth perturbation of 
$\Phi(x,t,\uy,\oy)=\overline{y}\left(\overline{x}+ \underline{x}^T J \underline{y}+t^2g(\tfrac{-\underline{y}}{t}) \right)$.

However, the size of the $x$ support of the amplitude $a_\ell$ is no longer uniform in $\ell$, since $a_\ell$ is supported in the set where $|\ux|\lesssim 2^{3\ell}$ and $|\ox|\lesssim 2^{4\ell}$. Hence we need to introduce spatial localization in the $x$ variable. To this effect, let $\Tilde{\rho}$ be a smooth, non-negative, compactly supported function in $\bbR^{2n+1}$ such that $\sum_{\mu\in\bbZ^{2n+1}}\Tilde{\rho}(x-\mu)=1$ for all $x\in\bbR^{2n+1}$. Define
\[\cT^{k,\ell}_{t,\mu}f(x):=\Tilde{\rho}(x-\mu)\int_{\R^{2n+1}} e^{i 2^{k-4\ell} \Phi_{\ell}(x,t,y)} a_{\ell}(x,t,y) f(\uy,\oy) dy.\]
Then we can write \[\cT^{k,\ell}_{t}f(x)=\sum_{\mu\in\bbZ^{2n+1}}\cT^{k,\ell}_{t,\mu}f(x).\]
The kernel of each $\cT^{k,\ell}_{t,\mu}$ is now supported in a set of constant size.
Since H\"ormander's estimate is stable under smooth perturbations, it can be applied to show that $\|\cT^{k,\ell}_{t,\mu}\|_{L^2(\bbR^{2n+1}}\lesssim 2^{-(k-4\ell)\frac{2n+1}{2}}$ uniformly in $\mu$ and $t$, provided we establish that
\begin{equation}
    \label{eq mixed Hessian np}
    \rank (\Phi_\ell)_{(x,y)}''= 2n+1
\end{equation}
with $\det (\Phi_\ell)_{(x,y)}''$ bounded from below uniformly in $\ell, t$ and $\mu$. Let 
\Be \label{eq:Xildef}
 \Xi^\ell(x,t,y):=\nabla_{x,t} \Phi_\ell (x,t,y).
 \Ee  
We calculate
 \[\Xi^\ell(x,t,y)=
 \overline{y}\begin{pmatrix}
 \frac{1}{2}J\uy+2^{-2\ell}t\nabla g_{\ell} (\tfrac{2^{-2\ell}\ux-\uy}{t})\\1\\  2tg(\tfrac{2^{-2\ell}\ux-\uy}{t})-\nabla g(\tfrac{2^{-2\ell}\ux-\uy}{t})\cdot(2^{-2\ell}\ux-\uy)+2^{-4\ell}\Upsilon_{\ell}(\tfrac{2^{-2\ell}\ux-\uy}{t})
 \end{pmatrix},
 \] where $\Upsilon_{\ell}(\tfrac{2^{-2\ell}\ux-\uy}{t}):=\frac{\partial}{\partial t}\left(t^2\Tilde{g}_\ell(\tfrac{2^{-2\ell}\ux-\uy}{t})\right)$ and satisfies \begin{equation}
    \label{eq Upsilon error}
     \left|\frac{\partial^\beta}{\partial \uw^{\beta}}\Upsilon_\ell(\uw)\right|\leq 4\cdot100n\cdot 2^{-4\ell}|\uw|^{8-|\beta|}\,\, \text{ for } |\beta|\in\{0,1,2,3\} \text{ and } |\uw|\leq\frac{1}{2}.
\end{equation} The partial derivative with respect to $t$ is not needed at the moment but will be helpful for calculations in the next subsection.
 Note that since $g$ is a homogeneous polynomial of degree $4$, we have 
 \[\nabla g(\tfrac{2^{-2\ell}\ux-\uy}{t})\cdot(2^{-2\ell}x-y)=4tg(\tfrac{2^{-2\ell}\ux-\uy}{t}).\]
Using the above in the last co-ordinate, and subscripts $1,2, \ldots, 2n,  2n+1$ to denote partial derivatives in $y_1, y_2, \ldots, y_{2n}$ and $\oy$ respectively, we have
\begin{align*}
 \Xi_{j}^{\ell}&=\oy\begin{pmatrix}
 \frac{1}{2}Je_j-2^{-2\ell}\partial_j\nabla g_{\ell}(\ttfl)\\0\\
 -2\partial_jg(\tfrac{2^{-2\ell}\ux-\uy}{t})+2^{-4\ell}t^{-1}\partial_j\Upsilon_\ell(\tfrac{2^{-2\ell}\ux-\uy}{t})
 \end{pmatrix},\\
 \Xi_{2n+1}^{\ell}&=\begin{pmatrix}
 \frac{1}{2}J\uy+2^{-2\ell}t\nabla g_{\ell} (\tfrac{2^{-2\ell}\ux-\uy}{t})\\1\\  -2tg(\tfrac{2^{-2\ell}\ux-\uy}{t})+2^{-4\ell}\Upsilon_{\ell}(\tfrac{2^{-2\ell}\ux-\uy}{t})
 \end{pmatrix}.
\end{align*}

Thus, for a fixed $t\in [1,2]$, the mixed Hessian $(\Phi_\ell)''_{(x,y)}$ is of the form
\begin{align*}
    \begin{pmatrix}
    \frac{J}{2}-2^{-2\ell}g''_\ell\left(\tfrac{2^{-2\ell}\ux-\uy}{t}\right)&*\\
    0&1
    \end{pmatrix}.
\end{align*}
Using \eqref{eq:Hessform}, we get
\[\frac{J}{2}-2^{-2\ell}g''\left(\tfrac{2^{-2\ell}\ux-\uy}{t}\right)=\frac{J}{2}-2^{-2\ell}\frac{u_\ell'(r)}{r}I_{2n}-2^{-2\ell}\left(\frac{u_\ell''(r)}{r^2}-\frac{u_\ell'(r)}{r^3}\right)\uw\uw^\intercal,\]
where $\uw= \tfrac{2^{-2\ell}\ux-\uy}{t}, r=|\uw|$ and $u_\ell$ satisfies \eqref{eq homerror}. We set $\sigma=-2^{-2\ell}\frac{u_\ell'(r)}{r}, \lambda=2^{-2\ell}\left(\frac{u_\ell''(r)}{r^2}-\frac{u_\ell'(r)}{r^3}\right), \kappa=\frac{1}{2}$, $\gamma=0$ and note that for $\ell\geq 400n$, $\max\{\sigma, \lambda, \kappa, \gamma, |\uw|\}\leq 1$, using \eqref{eq homerror} and the fact that $|\uw|\leq \frac{1}{2}$. Further $\sigma^2+\kappa^2\geq \tfrac{1}{4}$. Thus Corollary \ref{cor matrix inv est} implies that \[\det\left[\frac{J}{2}-2^{-2\ell}g''\left(\tfrac{2^{-2\ell}\ux-\uy}{t}\right)\right]\gtrsim_n 1.\]

\subsubsection{Almost orthogonality} 
\label{subsubsec almost L2}
We have thus verified \eqref{eq mixed Hessian np} and can conclude using H\"ormander's $L^2$ estimate that 
\begin{equation}
    \label{eq L2est cot-st piece}
    \|\cT^{k,\ell}_{t,\mu}\|_{L^2(\bbR^{2n+1})}\lesssim 2^{-(k-4\ell)\frac{2n+1}{2}}
\end{equation}
for each $\mu\in \bbZ^{2n+1}$. To avoid cumbersome notation, we suppress the dependence on $k, \ell$ and $t$ and denote $\cT^{k,\ell}_{t,\mu}$ by $\cT_\mu$. In order to obtain the required $L^2$ estimate for $\cT^{k,\ell}_{t}=\sum_{\mu\in\bbZ^{2n+1}}\cT_{\mu}$, by the Cotlar-Stein lemma, it suffices to show that
\begin{equation}
    \label{eq Cotlar St}
    \sup_{\mu,\nu\in\bbZ^{2n+1}}\{\|\cT_\mu \cT_\nu^*\|_{L^2}, \|\cT_\mu^*\cT_\nu\|_{L^2}\}\leq c_N 2^{-(k-4\ell)\frac{2n+1}{2}}(1+|\mu-\nu|)^{-N}
\end{equation}
for large enough $N$. Because of \eqref{eq L2est cot-st piece}, it is enough to consider the cases when $\mu\neq \nu$. The kernel of $\cT_\mu^*\cT_\nu$ is of the form
\[\int_{\R^{2n+1}} e^{i 2^{k-4\ell} [\Phi_{\ell}(x,t,y)-\Phi_{\ell}(x,t,\Breve{y})]} a_{\ell}(x,t,y)\overline{a_{\ell}(x,t,\Breve{y})}\Tilde{\rho}(x-\mu)\Tilde{\rho}(x-\nu) dx\]
and is non-zero only for $|\mu-\nu|\leq M$ for some constant $M$ depending on $\rho$. In view of \eqref{eq L2est cot-st piece}, this implies the desired estimate for $\|\cT_\mu \cT_\nu^*\|_{L^2}$ for a large enough constant $c_N$.

We now consider $\cT_\mu\cT_\nu^*$ whose kernel is given by
\[\cK_{\mu\nu^*}(x,\Breve{x}):=\Tilde{\rho}(x-\mu)\Tilde{\rho}(\Breve{x}-\nu) \int_{\R^{2n+1}} e^{i 2^{k-4\ell} [\Phi_{\ell}(x,t,y)-\Phi_{\ell}(\Breve{x},t)]} a_{\ell}(x,t,y)\overline{a_{\ell}(\Breve{x},t, y)}dy\]
which is supported in the set where $\max\left(|\ox|,|\Bar{\Breve{x}}|\right) \lesssim 2^{4\ell}$, $\max\left(|2^{-2\ell}\ux-\uy|,|2^{-2\ell}\underline{\Breve{x}}-\uy| \right)\leq \frac{1}{2}$, hence $|\ox-\Bar{\Breve{x}}|\lesssim 2^{4\ell}$ and $|\ux-\underline{\Breve{x}}|\leq 2^{2\ell-1}$. From the specific form of the phase function $\Phi_\ell$ \eqref{eq phase np Taylor}, we infer that
\[(\Phi_\ell)_{\ux\ux}''(x,t,y)=2^{-4\ell}\oy g_\ell''(\tfrac{2^{-2\ell}\ux-\uy}{t}),\]
while any second order derivative involving $\ox$ is zero. Using Taylor expansion of $\Phi_\ell$ around $x$ and differentiating with respect to $y$, we have
\begin{equation}
\label{eq cot Taylor np phase}
\nabla_y[\Phi_{\ell}(x,t,y)-\Phi_{\ell}(\Breve{x},t, y)]=(\Phi_{\ell})''_{x,y}(x,y,t)\cdot(x-\Breve{x})+ 2^{-4\ell}\Tilde{R}_\ell(\ux,\underline{\Breve{x}},t,y),
\end{equation}
with $\Tilde{R}_\ell$ satisfying the condition
\[|\Tilde{R}_\ell(\ux,\underline{\Breve{x}},t,y)|\lesssim |\ux-\underline{\Breve{x}}|^2\]
on account of \eqref{eq ggerror}.
Since $|\ux-\underline{\Breve{x}}|\lesssim 2^{3\ell}$ and $\Phi_\ell''(x,y,t)$ is invertible, for large enough $\ell$ we can conclude that
\[|\nabla_y[\Phi_{\ell}(x,t,y)-\Phi_{\ell}(\Breve{x},t, y)]|\geq |x-\Breve{x}|.\]
Integrating by parts $N$ times yields
\[|\cK_{\mu\nu^*}(x,\Breve{x})|\lesssim_N\Tilde{\rho}(x-\mu)\Tilde{\rho}(\Breve{x}-\nu)(1+2^{k-4\ell}|x-\Breve{x}|)^{-N}.\]
Thus
\[\int_{\bbR^{2n+1}}|\cK_{\mu\nu^*}(x,\Breve{x})|dx\lesssim_{N}2^{-(k-4\ell)N}\int_{|x-\Breve{x}|\sim|\mu-\nu|}\rho(x-\mu)|x-\Breve{x}|^{-N}dx\lesssim_N 2^{-(k-4\ell)N}(1+|\mu-\nu|)^{-N+2n}.\]
By symmetry we have the same estimate for $\int_{\bbR^{2n+1}}|\cK_{\mu\nu^*}(x,\Breve{x})|d\Breve{x}$ and thus it follows by the Schur test that
\[\|\cT_\mu^*\cT_\nu\|_{L^2}\lesssim_N 2^{-(k-4\ell)\frac{2n+1}{2}}(1+|\mu-\nu|)^{-N+2n}.\] The above estimate implies \eqref{eq Cotlar St} with $N$ replaced by $N-2n$, which is still good enough for large enough $N$ (say $N>4n+2$). An application of the Cotlar-Stein lemma concludes the proof of Proposition \ref{prop:L2cTkl-est}.

\subsection{Proof of Proposition \ref{prop kernel est np}}
\label{subsec ker est np}
Let $\Xi^\ell=\nabla_{x,t}\Phi_\ell$ as in \eqref{eq:Xildef},  $N\in\R^{2n+2}$ be a unit  vector, and let $\mathscr C^N(x,t,y)$ be the $(2n+1)\times (2n+1)$ cinematic curvature matrix with respect to $N$ given by
\Be\label{curvmatrix}  \mathscr C_{ij}^\ell= \frac{\partial^2}{\partial y_i\partial y_j} \inn{ N}{\Xi^\ell}  \Ee
We would like to argue as in the proof of \cite[Proposition 3.4]{MSS93} using a variant of  Stein's theorem  \cite{SteinBeijing} according to which the kernel estimate \eqref{eq cKkl L1Linf} holds provided the rank of the mixed Hessian $(\Phi_\ell)_{(x,t),y}''=2n+1$ and the curvature condition
\Be \label{eq:rank-curv}\inn {N}{ \Xi^\ell_{j}  }=0 ,\,\, j=1,\dots, 2n+1 \quad \implies \quad \rank\, \mathscr C^\ell = 2n\Ee 
is satisfied; i.e. the conic surface  $\Sigma^\ell_{x,t}$ parametrized by $y\mapsto \Xi^\ell(x,t,y)$  has the maximal number $2n$ of nonvanishing principal curvatures. The argument would be a consequence of the method of stationary phase applied to the oscillatory integral operator given by $\cT^{k,\ell}_t(\cT^{k,\ell}_{t'})^*$.

However, the kernel estimate in \cite{MSS93} was proved for a single oscillatory integral operator with a \textit{fixed} phase function $\Phi$, whereas we need to apply it for the family of operators $\cT^{k,\ell}_t$ with $\ell<\frac{k}{4}$. The implicit constant in Stein's theorem depends on the upper bound on finitely many spatial derivatives of the phase function $\Phi_\ell$ and the amplitude $a_\ell$, the size of the $y$ support of $a_\ell$ and the lower bound for the determinant of the invertible $2n\times2n$ minor of $\mathscr C^\ell$. 

We have already observed in the previous subsection that the spatial derivatives of the amplitude $a_\ell$ are bounded from above by a constant times the corresponding derivatives of $a$ for all large $\ell$. Further, arguing similarly as in the last subsection using property \eqref{eq ggerror} of $g_\ell$, it can be seen that $\Phi_\ell$ and $\nabla_{(x,t)}\Phi_\ell$ are smooth perturbations of $\Phi$ and $\nabla_{(x,t)}\Phi$, where 
$\Phi=\overline{y}\left(\overline{x}+ \underline{x}^T J \underline{y}+t^2g(\tfrac{-\underline{y}}{t}) \right)$.
Moreover, the size of the $y$ support of $a_\ell$ is bounded in $\ell$, and while the same is not true for the size of the $x$ support, arguing as in the previous subsection using \eqref{eq cot Taylor np phase}, we can still conclude that
\[|\nabla_y[\Phi_{\ell}(x,t,y)-\Phi_{\ell}(\Breve{x},\Breve{t}, y)]|\gtrsim |x-\Breve{x}|+|t-\Breve{t}|.\]
Since the estimate in \cite{MSS93} is stable under smooth perturbations, it suffices to show that $\rank{\mathscr{C}^\ell}=2n$ and the determinant of its invertible $2n\times 2n$ minor is bounded from below by an absolute constant independent of $\ell$.

To this effect, let $N=(\ua,\oa,\alpha_{2n+2})$ be a unit normal vector satisfying $\langle N, \Xi_j^\ell\rangle=0$ for $1\leq j\leq 2n+1$. The first $2n$ of these equations can be written down in the matrix form 
\[-\left(\frac{J}{2}-2^{-2\ell}g_{\ell}''\left(\frac{2^{-2\ell}\ux-\uy}{t}\right)\right)^{\intercal} \ua=\alpha_{2n+2}\left[-2\nabla g_{\ell}(\frac{2^{-2\ell}\ux-\uy}{t})+2^{-4\ell}\nabla\Upsilon_\ell(\frac{2^{-2\ell}\ux-\uy}{t})\right].\]

Recall that $\uw:=-\frac{2^{-2\ell}\ux-\uy}{t}$ and $r=|\uw|$. Let $\lambda_{\ell}(r):=\frac{u_{\ell}''(r)}{r^2}-\frac{u_{\ell}'(r)}{r^3}$ and $\sigma_{\ell}(r):=\frac{u_{\ell}'(r)}{r}$. Using \eqref{eq:Hessform} for $-\left(\frac{J}{2}-2^{-2\ell}g_{\ell}''\left(\frac{2^{-2\ell}\ux-\uy}{t}\right)\right)^{\intercal} =\left(\frac{J}{2}+2^{-2\ell}g_{\ell}''\left(\frac{2^{-2\ell}\ux-\uy}{t}\right)\right)$, we get
\begin{align*}
    \ua&=-\alpha_{2n+2}\left(\frac{J}{2}+2^{-2\ell}\sigma_{\ell}(r)I_{2n}+2^{-2\ell}\lambda_{\ell}(r)\uw\uw^\intercal\right)^{-1}(2\nabla g_{\ell}(\uw)-2^{-4\ell}\nabla \Upsilon_{\ell}(\uw)).
\end{align*}
Applying Lemma \ref{lem:radprop}, part (i) and Remark \ref{cor inn prod}  to calculate the gradient and matrix inverse on the right hand side yields
\begin{equation}
\label{eq oa}
\ua=\frac{-2\alpha_{2n+2}\sigma_{\ell}}{(\tfrac{1}{4}+2^{-4\ell}r^2\sigma_{\ell}\lambda_{\ell}+2^{-4\ell}\sigma_{\ell}^2)}\left(2^{-2\ell}\sigma_{\ell} I_{2n}-\frac{J}{2}\right)\uw +2^{-4\ell}\alpha_{2n+2}\Tilde{\uw},
\end{equation}
where $\Tilde{\uw}:=-\left(\frac{J}{2}+2^{-2\ell}\sigma_{\ell}(r)I_{2n}+2^{-2\ell}\lambda_{\ell}(r)\uw\uw^\intercal\right)^{-1}2^{-4\ell}\nabla \Upsilon_{\ell}(\uw)$. Using \eqref{eq Upsilon error} and Corollary \ref{cor matrix inv est}, it is easy to see that $|\Tilde{\uw}|\lesssim 2^{-4\ell}$. 

We compute 
\begin{align*}
 \Xi_{ij}^{\ell}&=\oy\begin{pmatrix}
2^{-2\ell}t^{-1}\partial^2_{ij}\nabla g_{\ell}(\uw)\\0\\ 2t^{-1}\partial^2_{ij}g_{\ell}(\uw)-2^{-4\ell}t^{-1} \partial^2_{ij}\Upsilon_{\ell}(\uw)
 \end{pmatrix}\,\, \text{ for } 1\leq i,j\leq 2n,\\
\Xi_{ij}^{\ell}&=\oy^{-1}\Xi_{j}^{\ell}\,\, \text{ for } i=2n+1.
\end{align*}
Using \eqref{eq gldef1} and \eqref{eq gldef2}, we can write
\begin{equation*}
 \Xi_{ij}^{\ell}=\oy\begin{pmatrix}
2^{-2\ell}t^{-1}\partial^2_{ij}\nabla g(\uw)\\0\\ 2t^{-1}\partial^2_{ij}g_{\ell}(\uw)
 \end{pmatrix}+\oy2^{-4\ell}\Tilde{\Xi}_{ij}^{\ell}(\uw)\,\, \text{ for } 1\leq i,j\leq 2n,
\end{equation*}
with $|\Tilde{\Xi}_{ij}^{\ell}(\uw)|\leq 4\cdot100n\cdot 2^{-4\ell}|\uw|^{6} $, using \eqref{eq Upsilon error}. 

We now use equation \eqref{eq oa} to calculate the elements of the cinematic curvature matrix:
\begin{align*}
&\frac{t}{\oy}\langle \Xi_{ij}^{\ell}, N\rangle \\
&=\frac{-2^{-2\ell+1}\alpha_{2n+2}\sigma_{\ell}}{\tfrac{1}{4}+2^{-4\ell}\sigma_{\ell}\lambda_{\ell}|\uw|^2+2^{-4\ell}\sigma_{\ell}^2}\left\langle\left(2^{-2\ell}\sigma_{\ell} I_{2n}-\frac{J}{2}\right)\uw,\partial_{ij}^2\nabla g(\uw) \right\rangle+2\alpha_{2n+2} \partial^2_{ij}g(\uw)+\alpha_{2n+2}2^{-4\ell}a_{ij}(\uw),
\end{align*}
where $|a_{ij}(\uw)|\lesssim 2^{-4\ell}$. Since $\langle\uw,\partial_{ij}^2\nabla g(\uw)\rangle=2\partial^2_{jk}g(\uw)$ and $\frac{1}{4}+2^{-4\ell}\sigma_{\ell}\lambda_{\ell}r^2+2^{-4\ell}\sigma_{\ell}^2=\frac{1}{4}+2^{-4\ell} u_{\ell}''\sigma_\ell$, we have
\begin{align*}
&\frac{t}{\oy}\langle \Xi_{ij}^{\ell}, N\rangle \\
&=\alpha_{2n+2}\left(2-\frac{2^{-4\ell+2}\sigma_{\ell}^2}{\tfrac{1}{4}+2^{-4\ell} u_{\ell}''\sigma_\ell}\right)\partial^2_{ij}g(\uw)+\frac{2^{-2\ell}\alpha_{2n+2}\sigma_{\ell}}{\frac{1}{4}+2^{-4\ell} u_{\ell}''\sigma_\ell}\left\langle J\uw,\partial_{jk}^2\nabla g(\uw) \right\rangle+\alpha_{2n+2}2^{-4\ell}a_{ij}(\uw).
\end{align*}

Because $u$ satisfies \eqref{eq homerror}, it follows that for sufficiently large $\ell$ (for example $\ell \geq 2$) and $|r|< \frac{1}{2}$, \[\max\left(\frac{\sigma_{\ell}(r)}{\tfrac{1}{4}+2^{-4\ell}u_{\ell}''\sigma_{\ell}},\sigma_{\ell}(r)\right)\lesssim 1.\]
Thus, for $1\leq i,j\leq 2n$ and $|\uw|<\frac{1}{2}$,
\begin{equation}
    \label{eq cinmatrix lower bound}
    \left|\frac{t}{\oy}\langle \Xi_{ij}^{\ell}, N\rangle-2\alpha_{2n+2}\partial^2_{ij}g(\uw)\right|\lesssim 2^{-2\ell}.
\end{equation}
From $\eqref{eq oa}$, we infer that $|\oa|\lesssim |\alpha_{2n+2}|$. This, combined with the equation $\langle N, \Xi_{2n+1}\rangle=0$, further implies that $|\oa|\lesssim |\alpha_{2n+2}|$. Since $N$ is a unit vector, it follows that $|\alpha_{2n+2}|\gtrsim 1$.

The second part of Lemma \ref{lem:radprop} implies that 
\[\det g''(\uw)=\left(\frac{2r^3}{r}\right)^{2n-1}6r^2,\]
which is uniformly bounded away from $0$, since $\frac{1}{8}\leq r=|\uw|\leq \frac{1}{2}$ when $\uw=\frac{2^{-2\ell}\ux-\uy}{t}$ is in the support of $a_\ell$ for $\ell<\frac{k}{4}$. Hence, we conclude that $g''(\uw)$ is invertible. 

Let $\bar{\mathscr{C}}_{\ell}(x,t,y)$ denote the $2n\times2n$ matrix with the $(i,j)$-th entry given by $\langle \Xi_{ij}^{\ell}, N\rangle$ for $1\leq i,j\leq 2n$. Using \eqref{eq cinmatrix lower bound}, we conclude that $\bar{\mathscr{C}}_{\ell}(x,t,y)$ is invertible as well, with its determinant bounded from below by an absolute constant independent of $\ell$ for $\ell\geq 400n$.

Further, observe that for $i=2n+1$, $\langle N,\Xi_{ij}\rangle=\langle N, \oy^{-1}\Xi_{j}\rangle=0$ for $1\leq j\leq 2n+1$. Since the cinematic curvature matrix at $(x,t,y)$ takes the form
\[\mathscr{C}^\ell(x,t,y)=2\frac{\alpha_{2n+2}}{\oy}\begin{pmatrix}
\bar{\mathscr{C}}_\ell& & 0\\
0& & 0
\end{pmatrix},\]
we conclude that it is of rank $2n$. This finishes the proof of Proposition \ref{prop kernel est np}.

\section{Proof of Propositions \ref{prop:L2osc-est} and \ref{prop kernel est eq}}
\label{sec eq im}
In this section, we prove the fixed time $L^2$ estimate \eqref{prop:L2osc-est} for the operator $T^k_t$ and the kernel estimate \eqref{eq Kk L1Linf} for the kernel $K^k_{tt'}$ corresponding to the equatorial and intermediate regions. We shall also observe that the former estimate is stable in the intermediate region, thus obtaining the $L^2$ estimate for $T^k$ with kernel supported around the north pole as well. We shall use H\"ormander's $L^2$ estimate \cite[ch. IX.1]{Stein-harmonic} and the curvature condition in \cite[Proposition 3.4]{MSS93} again to prove Propositions \ref{fix time Tkest} and \ref{prop kernel est eq} respectively. Since we use them for a fixed oscillatory operator (and not a family of them as in the previous section), these results are directly applicable.

\subsection{Estimates around the Equator}
Recall that the phase function of $T_k$ around the equator is given by $\Phi=\Phi^\eq=y_1H_1(x,t,y',\oy)$ with
\[\nabla_{(x,t)}H_1(x,t,y',\oy)=-\left[\left(\frac{\partial F}{\partial y_1}\right)^{-1}\nabla_{x,t}F\right]\bigg|_{(x,t,H_1(x,t,y',\oy), y', \oy)}.\]
Here $F(x,t,y)$ is the defining function of the Kor\'anyi sphere centred at $x$ and of radius $t$ 
\[F(x,t,y)=|\ux-\uy|^4+|\ox-\oy+\tfrac{1}{2}\ux^{\intercal}J\uy|^2-t^4.\]
The amplitude $b=b_\eq(x,t,y)$ is compactly supported in a set where $x$ is small, $|y_1|\in[1/2,2]$, $t^{-1}|x'-y'|<2^{-400n}$, $t^{-2}|\ox-\oy+\frac{1}{2}\ux^{\intercal}J(H_1(x,t,y',\oy),y')|<2^{-400n}$ and $t^{-1}|x_1-H_1(x,t,y',\oy)|>1-2^{-400n}$.
\subsubsection{Proof of Proposition \ref{prop:L2osc-est} at the Equator}
\label{subsubsec L2 eq}
By H\"ormander's $L^2$ estimate, it suffices to prove that the mixed Hessian $\Phi_{x,y}''$ is of rank $2n+1$ in the support of $b$, with the determinant bounded from below by an absolute constant independent of $t$. 

Let $\Theta(x,y):=(\ux-\uy, \ox-\oy+\tfrac{1}{2}\ux^\intercal J \uy)$. Since $\det \Phi_{x,y}''(x,t,y)=\det\Phi_{x,y}''(0,t, \Theta(x,y))$ (see \S\ref{section trans invar}) , it suffices to verify the curvature condition under the assumptions $x=0$, $|\oy|\in [1/2,2]$, $t^{-1}|y'|<2^{-400n}$, $t^{-2}|\oy|<2^{-400n}$ and $t^{-1}|H_1(0,t,y',\oy)|>t(1-2^{-400n})$.

Let \[\Xi(x,t,y)=\nabla_{(x,t)}\Phi^\eq(x,t,y).\]
We calculate
\[\Xi(0,t,y)=y_1\nabla_{(x,t)}H_1(0,t,y',\oy)=y_1\left[\frac{1}{4|\uy|^2y_1}(4|\uy|^2\uy+\oy(J\uy+2e_{2n+1})-4t^3e_{2n+2})\right]\Big|_{|y_1=H_1(0 ,t,y',\oy)}.\]
Let $\Xi_1:=\frac{\partial}{\partial{y_1}}\Xi=\nabla_{(x,t)}H_1(0,t,y',\oy)$.
\begin{rem}
\label{rem simplify vec}
We have
\[\nabla_{(x,t)}H_1(x ,t,y',\oy)=-\left[\left(\frac{\partial F}{\partial y_1}\right)^{-1}\nabla_{x,t}F\right]\bigg|_{y_1=H_1(x,t,y',\oy)}.\]
Taking partial derivatives with respect to $y_j$ and using the chain rule, we obtain
\begin{align*}
\frac{\partial}{\partial{y_j}}\nabla_{(x,t)}H_1(0,t,y')&=-\left[\left(\frac{\partial F}{\partial y_1}\right)^{-1}\left(\frac{\partial}{\partial y_j}+\frac{\partial H_1}{\partial y_j}\frac{\partial}{\partial y_1}\right)\nabla_{x,t}F\right]\bigg|_{x=0, y_1=H_1(0 ,t,y',\oy)}\\
&+ \left[\nabla_{x,t}F\left(\frac{\partial F}{\partial y_1}\right)^{-2}\left(\frac{\partial}{\partial y_j}+\frac{\partial H_1}{\partial y_j}\frac{\partial}{\partial y_1}\right)\frac{\partial F}{\partial y_1}\right]\bigg|_{x=0, y_1= H_1(0 ,t,y',\oy)}
\end{align*}
for $2\leq j\leq 2n+1$. Observe that the second term is a scalar multiple of $\Xi_1=\nabla_{(x,t)}H_1(0,t,y',\oy)$. Define \[\Xi_j:=-\left[\left(\frac{\partial F}{\partial y_1}\right)^{-1}\left(\frac{\partial}{\partial y_j}+\frac{\partial H_1}{\partial y_j}\frac{\partial}{\partial y_1}\right)\nabla_{x,t}F\right]\bigg|_{x=0, y_1=H_1(0 ,t,y',\oy)}\] for $2\leq j\leq 2n+1$. Then it follows that the tangent space of $\Xi$ at $(x,t,y)$ is spanned by $\Xi_1, \ldots, \Xi_{2n+1}$.
We shall repeatedly use this kind of observation to simplify the expressions for tangent vectors and their derivatives. 
\end{rem}
Using the fact that \[
\frac{\partial H_1(0,t,y')}{\partial{y_j}}=-\frac{y_j}{H_1(0,t,y',\oy)},\,\,\frac{\partial H_1(0,t,y')}{\partial{\oy}}=-\frac{\oy}{2|\uy|^2y_1}\bigg|_{y_1=H_1(0,t,y')},
\]
we calculate
\begin{align*}
    \Xi_1&=\nabla_{(x,t)}H_1(0,t,y',\oy)=\left[\frac{1}{4|\uy|^2y_1}(4|\uy|^2\uy+\oy(J\uy+2e_{2n+1})-4t^3e_{2n+2})\right]\bigg|_{y_1=H_1(0 ,t,y',\oy)},\\
    \Xi_j&=y_1\left[\frac{1}{4|\uy|^2y_1}\left(8y_j\uy+(4|\uy|^2+\oy J)e_j-\frac{y_j}{y_1}\left(8y_1\uy+(4|\uy|^2+\oy J)e_1\right)\right)\right]\bigg|_{y_1=H_1(0 ,t,y',\oy)},\,\, 2\leq j\leq 2n,\\
     \Xi_{2n+1}&=y_1\left[\frac{1}{4|\uy|^2y_1}\left(J\uy+2e_{2n+1}-\frac{\oy}{2|\uy|^2y_1}\left(8y_1\uy+(4|\uy|^2+\oy J)e_1\right)\right)\right]\bigg|_{y_1=H_1(0 ,t,y',\oy)}.
\end{align*}

Henceforth, $I$ shall denote the $2n\times 2n$ identity matrix $I_{2n}$. The above discussion implies that the mixed Hessian $\Phi_{(x,y)}''$ at $x=0$ and $y'=0$ has the same determinant as the matrix
\begin{equation}
    \label{eq rot matrix eq}
    y_1\left[\frac{1}{4|\uy|^2y_1}\begin{pmatrix}
I_{y_1}\left(4|\uy|^2I-\oy J\right)& &2\oy e_1\\
& & \\
-\frac{\oy}{2|\uy|^2y_1}\left(12|\uy|^2e_1+\oy Je_1\right)^{\intercal}+(J\uy)^{\intercal}& & 2
\end{pmatrix}\right]\Bigg|_{y_1=H_1(0,t,0,\oy), y'=0}.
\end{equation}
Here $I_{y_1}$ is the $2n\times 2n$ diagonal matrix with $y_1$ as the first diagonal entry and the rest being $1$. Henceforth, we shall abbreviate $H_1(0,t,0,\oy)$ to $H_1$. The determinant of the above matrix is equal to
\[y_1^{2n+1}\left(\left[\frac{1}{4|H_1|^2H_1}\right]^{2n+1}\det\left(4|H_1|^2I-\oy J\right)(2-\gamma)\right),\]
with $\gamma:=\left[\left(-\frac{\oy}{2|\uy|^2y_1}\left(12|\uy|^2e_1+\oy Je_1\right)+J\uy\right)^{\intercal}\left(4|\uy|^2I-\oy J\right)^{-1}(\frac{2\oy}{y_1}e_1)\right]\Big|_{y_1=H_1,y'=0}$.

We use Lemma $\ref{lem matrix inv}$ to obtain $\left(4|\uy|^2I-\oy J\right)^{-1}=\frac{1}{16|\uy|^4+|\oy|^2}\left(4|\uy|^2I+\oy J\right)$. Recalling that $|\uy|=|H_1e_1|=|H_1|$, we calculate
\begin{align*}
\gamma&=
\frac{2\oy}{H_1(16|H_1|^4+|\oy|^2)}\left(-\frac{1}{2|H_1|^2H_1}\left(12|H_1|^2\oy e_1+(|\oy|^2-2|H_1|^4)Je_1\right)\right)^{\intercal}\left(4|H_1|^2e_1+\oy Je_1\right)\\
&=-\frac{\oy^2}{H_1^4(16|H_1|^4+|\oy|^2)}\left(48|H_1|^4-2|H_1|^4+|\oy|^2\right)\leq 0.
\end{align*}
Thus the absolute value of the determinant of the matrix in \eqref{eq rot matrix eq} greater than or equal to 
\[|y_1|^{2n+1}\det\left[\frac{1}{2|H_1|^3}\left(4|H_1|^2I-\oy J\right)\right].\]
Recall that $1>t^{-1}|H_1|>1-2^{-400n}$ and $t^{-2}|\oy|<2^{-400n}$. Thus we have
$16|H_1|^4+|\oy|^2\geq 1$. Using Corollary \ref{cor matrix inv est} and the fact that $|y_1|\sim 1$, we conclude that 
\[|y_1|^{2n+1}\det\left[\frac{1}{2|H_1|^3}\left(4|H_1|^2I-\oy J\right)\right]\gtrsim 1,\]
and the implicit constant is independent of $t\in [1,2]$. It follows that the same is true for $\Phi_{x,y}''$ or in other words, the mixed Hessian of $\Phi$ is of rank $2n+1$. This concludes the proof of Proposition \ref{prop:L2osc-est} around the equator.
\subsubsection{Proof of Proposition \ref{prop kernel est eq} around the Equator}
\label{subsubsec kernel eq}
Let $N=(\ua,\oa,\alpha_{2n+2})$ be a unit normal vector such that $\langle N, \Xi_j\rangle=0$ for $1\leq j\leq 2n+1$. Let $\mathscr C(x,t,y)$ be the $(2n+1)\times (2n+1)$ cinematic curvature matrix with respect to $N$ given by
\begin{equation}
\label{eq curv matrix eq}
\mathscr C_{ij} = \frac{\partial^2}{\partial y_i\partial y_j} \inn{ N}{\Xi}    
\end{equation}  
We argue as in the proof of \cite[Proposition 3.4]{MSS93} according to which Proposition \ref{prop kernel est eq} holds provided that $\rank{\Phi_{(x,t),y}''}=2n+1$  and  the additional  curvature condition
\Be \label{eq:rank-curv eq}\inn {N}{ \Xi_{j}  }=0 ,\,\, j=1,\dots, 2n+1 \quad \implies \quad \rank\, \mathscr C = 2n\Ee 
is satisfied; i.e. the conic surface  $\Sigma_{x,t}$ parametrized by $y\mapsto \Xi(x,t,y)$  has the maximal number $2n$ of nonvanishing principal curvatures. 
It remains to verify \eqref{eq:rank-curv eq}.

Since $\rank{\mathscr{C}(x,t,y)}=\rank{\mathscr{C}(0,t,\Theta(x,y))}$, we may again assume that $x=0$ (see \S\ref{section trans invar}). Further, when $y_i=0$ for $2\leq i\leq 2n$, the equations $\langle N, \Xi_j\rangle=0$ for $1\leq j\leq 2n$ can be expressed as
\begin{equation*}
\left(4|\uy|^2I-\oy J\right)\ua\big|_{y_1=H_1}= -\frac{1}{y_1}(2\oy\oa-4t^3\alpha_{2n+2})e_1:=Te_1,
\end{equation*}
Solving, we get \begin{equation}
    \label{eq equator cin1}
    \ua=T\left(4|\uy|^2I-\oy J\right)^{-1}e_1\big|_{y_1=H_1}=\frac{T}{16|\uy|^4+|\oy|^2}\left(4|\uy|^2I+\oy J\right)e_{1}\Big|_{y_1=H_1}.
\end{equation} 
\begin{rem}
\label{rem N bdd below}
Equation \eqref{eq equator cin1} implies that $|\ua|\lesssim T$. Further, since $\left|\frac{\oy}{2|\uy|^2H_1}\right|\lesssim 2^{-400n}$, 
the equation $\langle N, \Xi_{2n+1}\rangle=0$ implies that $|\oa|\lesssim |\ua|$ and thus $|\oa|\lesssim T$ as well. Using the definition of $T$, we have that $T\gtrsim |\alpha_{2n+2}|-2^{-400n}|\oa|$ and therefore $|\alpha_{2n+2}|\lesssim T$. We conclude that $|N|\lesssim T$ and since $N$ is a unit vector, it follows that $|T|\gtrsim 1$.
\end{rem}

We now calculate the derivatives of the tangent vectors, using subscripts to denote the corresponding partial derivatives. Observe that $\Xi_{1j}$ is a scalar multiple of $\Xi_j$ for $2\leq j\leq 2n+1$, while $\Xi_{11}=0$. Thus $\langle \Xi_{1j},N\rangle=0$ for $1\leq j\leq 2n+1$. 

Let $\delta_{ij}:=\delta_0(i-j)$ where $\delta_0$ denotes the Dirac measure at the origin in $\bbR$. Recall that $H_1=H_1(0,t,0,\oy)$ and thus $|\uy|=|H_1|$ when $y'=0$. For $2\leq i,j\leq 2n+1$, we have
\begin{align*}
\langle \Xi_{ij}, N\rangle&=\delta_{ij}\frac{y_1}{4|\uy|^2y_1}\left[8\uy-\frac{1}{y_1}\left(8y_1\uy+(4|\uy|^2I+\oy J)e_1\right)\right]^{\intercal}\bigg|_{y_1=H_1}\ua\\
&=-\delta_{ij}\frac{y_1}{4|H_1|^2H_1^2}\left(4|H_1|^2I+\oy J)e_1\right)^{\intercal}\ua.    
\end{align*}

Here we have used that $y_j=0$ for $2\leq j\leq 2n$ and also the fact that the dot product of $N$ with terms involving the first derivatives of $\nabla_{(x,t)}F$ is zero since they lie in the tangent space spanned by vectors $\Xi_1, \ldots, \Xi_{2n+1}$.
Using \eqref{eq equator cin1} and the properties $J^2=-I_{2n}$ and $J^\intercal=-J$, we calculate
\[
\langle \Xi_{ij}, N\rangle= -y_1\delta_{ij}\frac{1}{4|H_1|^2H_1^2}\frac{T}{16|H_1|^4+|\oy|^2}({16|H_1|^4+|\oy|^2})
=-y_1\delta_{ij}\frac{T}{4|H_1|^4}.
\]
For $i=2n+1$ and $2\leq j\leq 2n$, we have 
\begin{align*}
    \langle \Xi_{ij},N\rangle&=\frac{y_1}{4|H_1|^2H_1}\left[Je_j-\frac{\oy}{2|H_1|^2H_1}(8H_1e_j)\right]^{\intercal}\ua\\
    &=\frac{y_1}{4|H_1|^3}\left[\frac{4T}{|H_1|^2(16|H_1|^4+|\oy|^2)}\left(|H_1|^4+|\oy|^2\right)\right]e_1Je_j
\end{align*}
and 
\begin{align*}
     \langle \Xi_{ii},N\rangle&=\frac{y_1}{4|H_1|^2H_1}\left(-\frac{1}{2|H_1|^2H_1}\right)\left[\frac{2|H_1|^4+\oy^2}{|H_1|^2}\left(6e_1+\frac{3\oy}{2|H_1|^2}Je_1\right)\right]^{\intercal}N\\
     &=\frac{y_1}{4|H_1|^2H_1}\left(-\frac{3T}{4|H_1|^4H_1}\right)(2|H_1|^4+\oy^2).
\end{align*}     
Let $P:\bbR^{2n}\to \bbR^{2n-1}$ be the projection onto the orthogonal complement of $e_1$. At the equator, when $x=0$, $y'=0$ and $\oy=0$, we have $H_1=t$ and the cinematic curvature matrix 
with respect to $N$ is given by
\[\mathscr C\Big|_{x=0,y_1=te_1,\oy=0}=\frac{T}{4t^3}\begin{pmatrix}
0& & 0& &0\\
0& & -I_{2n-1}& & \frac{1}{4t}Pe_1Je_j\\
0& &  \frac{1}{4t}(e_1Je_j)^{\intercal}P^\intercal& & -\frac{3}{2}
\end{pmatrix},\]
Since $|T|\gtrsim 1$, the rank of the above matrix is $2n$, provided $\left|-\frac{3}{2}+\frac{1}{16t^2}|e_1J|^2\right|\neq 0$. This is true, for \[\frac{1}{16t^2}|e_1J|^2=\frac{1}{16t^2}\leq \frac{1}{16}\]
when $t\in [1,2]$. This verifies \eqref{eq:rank-curv eq}, and finishes the proof of Proposition \ref{prop kernel est eq} at the equator.

\subsection{Estimates in the Intermediate Region}
Recall that the phase function of $T_k$ in the intermediate region is given by $\Phi=\Phi^\im=\oy\oH(x,t,\uy)$ with
\[\nabla_{(x,t)}\oH(x,t,y',\oy)=-\left[\left(\frac{\partial F}{\partial \oy}\right)^{-1}\nabla_{x,t}F\right]\bigg|_{(x,t,\oH(x,t,\uy), \uy)},\]
where $F(x,t,y)$ is the defining function of the Kor\'anyi sphere of radius $t$.
The amplitude $b=b_\im$ is supported in a small neighborhood of the set where $x$ is small, $|\oy|\in [1/2,2]$, $t^{-1}|x'-y'|<2^{-400n}$ and
$\min\left(t^{-1}|x_1-y_1|, t^{-2}|\ox-\oH(x,t,\uy)+\tfrac{1}{2}\ux^{\intercal}J\uy|\right)>2^{-400n}$.
\subsubsection{Proof of Proposition \ref{fix time Tkest} in the Intermediate region}
\label{subsubsec IM L2}
By H\"ormander's $L^2$ estimate, it again suffices to prove that the mixed Hessian $\Phi_{x,y}''$ is of rank $2n+1$ in the support of $b$, with the determinant bounded from below by an absolute constant independent of $t$. 

Because of the fact that $\det \Phi_{x,y}''(x,t,y)=\det\Phi_{x,y}''(0,t, \Theta(x,y))$ (with $\Theta(x,y):=(\ux-\uy, \ox-\oy+\tfrac{1}{2}\ux^\intercal J \uy)$, see \S\ref{section trans invar}), it again suffices to verify the curvature condition under the assumptions that $x=0$, $|\oy|\in [1/2,1]$, $t^{-1}|y'|<2^{-400n}$ and $\min\left(t^{-1}|y_1|, t^{-2}|\oH(x,t,\uy)|\right)>2^{-400n}$.

Let \[\Xi(x,t,y)=\nabla_{(x,t)}\Phi^\eq(x,t,y).\]
We calculate
\[\Xi(0,t,y)=\oy\nabla_{(x,t)}\oH(0,t,y',\oy)=\oy\left[\frac{1}{2\oy}(4|\uy|^2\uy+\oy J\uy+2\oy e_{2n+1})-4t^3e_{2n+2})\right]\bigg|_{\oy=\oH(0 ,t, \uy)}.\]
The tangent vectors at $x=0$ are spanned by
\begin{align*}
    \Xi_j&=\frac{\oy}{2\oH(0,t,\uy)}\left[8y_j\uy+(4|\uy|^2+\oH(0,t,\uy)J)e_j-\frac{4|\uy|^2y_j}{\oH(0,t,\uy)}\left(\frac{1}{2}J\uy+e_{2n+1}\right)\right],\,\, 1\leq j\leq 2n,\\
     \Xi_{2n+1}&=\nabla_{(x,t)}H(0,t,\uy).
\end{align*} Here we have used the same idea as in Remark \ref{rem simplify vec} to get a simplified form for the tangent vectors.
We abbreviate $\oH(0,t,y_1e_1)=\oH$. The mixed Hessian of $\Phi$ when $x=0$ and $y_j=0$ for $2\leq j\leq 2n$ has determinant equal to that of the matrix
\[
\frac{\oy}{\oH}\begin{pmatrix}
4\uy\uy^{\intercal}+2|\uy|^2I-\frac{\oH}{2} J-\frac{|\uy|^2}{\oH}\uy(J\uy)^{\intercal}& &-\frac{2|\uy|^2}{\oH} \uy\\
& & \\
\frac{1}{\oH}\left[(2|\uy|^2I+\frac{\oH}{2} J)\uy\right]^{\intercal}& & 1
\end{pmatrix}.\]

Using Schur's complement, the absolute value of the determinant of the above matrix is 
\[\left|\frac{\oy}{\oH}\right|^{2n+1}\det X \left[1-\frac{1}{\oH}\left(2|\uy|^2I+\frac{\oH}{2} J)\uy\right)^{\intercal}X^{-1}\left(-\frac{2|\uy|^2}{\oH}\uy\right)\right],\]
where
\[X=4\uy\uy^{\intercal}+2|\uy|^2I-\frac{\oH}{2} J-\frac{|\uy|^2}{\oH}\uy(J\uy)^{\intercal}.\]
Using Corollary \ref{cor inn prod} with $\sigma=2|\uy|^2, \lambda=4, \kappa=-\frac{\oH}{2}$ and $\gamma=-\frac{|\uy|^2}{\oH}$, we get $ X^{-1}\uy=\frac{4}{46|\uy|^4+|\oH|^2}(2|\uy|^2I+\frac{\oH}{2} J)\uy$. 
Thus
\[-\frac{1}{\oH}\left(2|\uy|^2I+\frac{\oH}{2} J)\uy\right)^{\intercal}X^{-1}\left(-\frac{2|\uy|^2}{\oH}\uy\right)=\frac{2|\uy|^2}{\oH^2}\frac{16|\uy|^4+|\oH|^2}{46|\uy|^4+|\oH|^2}|\uy|^2\geq 0.\]
Hence \[\left|\det X \left[1-\frac{1}{\oH}\left(2|\uy|^2I+\frac{\oH}{2} J)\uy\right)^{\intercal}X^{-1}\left(-\frac{2|\uy|^2}{\oH}\uy\right)\right]\right|\geq |\det X|.\]
Using Corollary $\ref{cor matrix inv est}$ for $X$ with 
$\sigma^2+\kappa^2=4|\uy|^4+\frac{|\oH|^2}{4}\geq \frac{t^4}{4}\geq \frac{1}{4}$, we conclude that $|\det X|\gtrsim_n 1$ uniformly in $t\in[1,2]$. Since $\left|\frac{\oy}{\oH}\right|^{2n+1}\gtrsim 1$ in the support of $b=b_\im$, it follows that $\det\Phi_{x,y}''\gtrsim_n 1$ uniformly in $t\in[1,2]$ as well. In other words, the mixed Hessian has rank $2n+1$, which concludes the proof.

\subsubsection{Proof of Proposition \ref{prop kernel est eq} in the Intermediate region}
\label{subsubsec IM kernel}
Let $N=(\ua,\oa,\alpha_{2n+2})$ be a unit normal vector satisfying $\langle N, \Xi_j\rangle=0$ for $1\leq j\leq 2n+1$. Let $\mathscr C(x,t,y)$ be the $(2n+1)\times (2n+1)$ cinematic curvature matrix with respect to $N$. To prove Proposition \ref{prop kernel est eq} in the intermediate region, we need to verify the additional curvature condition \eqref{eq:rank-curv eq} for $\mathscr C(x,t,y)$. We may again assume without loss of generality that $x=0$.

When $y_i=0$ for $2\leq i\leq 2n$, the equations $\langle N, \Xi_j\rangle=0$ for $1\leq j\leq 2n$ can be written down in the matrix form as 
\begin{equation*}
X\ua=\left(4\uy\uy^{\intercal}+2|\uy|^2I-\frac{\oH}{2} J-\frac{|\uy|^2}{\oH}\uy(J\uy)^{\intercal}\right)\ua= \oa\frac{2|\uy|^2}{\oH} \uy.
\end{equation*}
Thus \begin{equation}
    \label{eq equator cin2}
    \ua=\oa\frac{2|\uy|^2}{\oH} X^{-1}\uy=\frac{2\oa|\uy|^2}{\oH(46|\uy|^4+|\oH|^2)}(8|\uy|^2I+2\oH J)\uy.
\end{equation} 
\begin{rem}
Arguing as in Remark \ref{rem N bdd below}, using the equations above and the fact that $|N|=1$, it is not difficult to see that $|\oa|\gtrsim 1$.
\end{rem}
We now calculate the derivatives of the tangent vectors, denoting them by subscripts. Observe that for $i=2n+1$, $\Xi_{ij}$ is a scalar multiple of $\Xi_j$ for $1\leq j\leq 2n$, while $\Xi_{ii}=0$. Thus $\langle \Xi_{ij},N\rangle=0$ for $1\leq j\leq 2n+1$. Further, for $2\leq i,j\leq 2n$, we have
\[\langle \Xi_{ij}, N\rangle=\frac{\oy}{\oH}\delta_{ij}\left[4\uy-\frac{2|\uy|^2}{\oH}\left(\frac{1}{2}J\uy+e_{2n+1}\right)\right]^{\intercal}N.
\]
Here we use the fact that $y_j=0$ for $2\leq j\leq 2n$ and the dot product of $N$ with terms involving the first derivatives of $\nabla_{(x,t)}F$ disappears since these terms lie in the space spanned by the tangent vectors.
For $2\leq i\leq 2n$, using \eqref{eq equator cin2}, we calculate
\[
\langle \Xi_{ij}, N\rangle= \frac{\oy}{\oH}\delta_{ij}\frac{2\oa|\uy|^2}{\oH}\left(\frac{30|\uy|^4}{46|\uy|^4+|\oH|^2}-1\right)=-\frac{\oy}{\oH}\delta_{ij}\frac{2\oa|\uy|^2}{\oH}\left(\frac{16|\uy|^4+|\oH|^2}{46|\uy|^4+|\oH|^2}\right).
\]
For $i=1$ and $2\leq j\leq 2n$, using the fact that $\frac{\partial }{\partial y_1}\oH(0,t,y_1e_1)=-\frac{2|\uy|^2y_1}{\oH}$, we obtain
\begin{align*}
    \langle \Xi_{ij},N\rangle&=\frac{\oy}{\oH}\left[4y_1e_j-\frac{|\uy|^2y_1}{\oH}Je_j\right]^{\intercal}N=-\frac{\oy}{\oH}\frac{16\oa|\uy|^2}{|\oH|^2}y_1\uy Je_j\left(\frac{|\oH|^2+|\uy|^4}{46|\uy|^4+|\oH|^2}\right)
\end{align*}  
and
\begin{align*}
    \langle \Xi_{ii},N\rangle&=\frac{\oy}{\oH}\left[12\uy-\frac{1}{2}\left(\frac{10|\uy|^2}{\oH}+\frac{4|\uy|^6}{|\oH|^3}\right)J\uy-\frac{1}{\oH}\left(6|\uy|^2+\frac{4|\uy|^6}{|\oH|^2}\right)e_{2n+1}\right]^{\intercal}N\\
     &=\frac{2\oa\oy|\uy|^2}{\oH^2}\left[\frac{96|\uy|^4}{{46|\uy|^4+|\oH|^2}}-\frac{|\uy|^4}{|\oH|^2}\left(\frac{10|\oH|^2+4|\uy|^4}{46|\uy|^4+|\oH|^2}\right)-\frac{3|\oH|^2+2|\uy|^4}{|\oH|^2}\right].
    %
    %
    %
\end{align*}
Simplifying, we get
\[\frac{96|\uy|^4}{{46|\uy|^4+|\oH|^2}}-\frac{|\uy|^4}{|\oH|^2}\left(\frac{10|\oH|^2+4|\uy|^4}{46|\uy|^4+|\oH|^2}\right)-\frac{3|\oH|^2+2|\uy|^4}{|\oH|^2}=\frac{12|\uy|^4}{{46|\uy|^4+|\oH|^2}}\left(7-8\frac{|\uy|^4}{|\oH|^2}\right)-3:=D.\]
%
Thus the cinematic curvature matrix at $y=y_1e_1+\oH e_{2n+1}$ is given by
\begin{equation}
    \label{eq cin matrix IM}
    \frac{2\oa\oy|\uy|^2}{\oH^2}\begin{pmatrix}
D& & B^{\intercal}P^\intercal& &0\\
PB& & -\lambda I_{2n-1}& & 0\\
0& &0 & & 0
\end{pmatrix},
\end{equation}
where $\lambda=\frac{16|\uy|^4+|\oH|^2}{46|\uy|^4+|\oH|^2}$, $B=-\frac{8y_1}{\oH}\left(\frac{|\oH|^2+|\uy|^4}{46|\uy|^4+|\oH|^2}\right)\uy^{\intercal}J$, $P:\bbR^{2n}\to\bbR^{2n-1}$ is the projection onto the orthogonal complement of $e_1$ and $D$ is as defined above. Since $|\uy|^4+|\oH|^2=t^4$ with $t\in [1,2]$, we have that $\frac{1}{46}\leq\lambda\leq1$. Further, $\min(|\uy|, |\oy|, |\oH|)\geq 2^{-400n}$ and $|\oH|\leq 4$.
Thus, the rank of the above matrix is $2n$, provided 
$|D+\lambda^{-1} B^{\intercal}B|$ is uniformly bounded from below by a positive constant. 
We have
\[|D+\lambda^{-1}B^{\intercal}B|\geq 3-\frac{12|\uy|^4}{{46|\uy|^4+|\oH|^2}}\left(7-8\frac{|\uy|^4}{|\oH|^2}\right)-\frac{64|y_1|^2}{|\oH|^2}\left(\frac{(|\oH|^2+|\uy|^4)^2}{(46|\uy|^4+|\oH|^2)(16|\uy|^4+|\oH|^2)}\right)|\uy|^2.\]
Now
\begin{align*}
     &\frac{12|\uy|^4}{{46|\uy|^4+|\oH|^2}}\left(7-8\frac{|\uy|^4}{|\oH|^2}\right)+\frac{64|y_1|^2}{|\oH|^2}\left(\frac{(|\oH|^2+|\uy|^4)^2}{(46|\uy|^4+|\oH|^2)(16|\uy|^4+|\oH|^2)}\right)|\uy|^2\\
    &\leq \frac{12|\uy|^4}{{46|\uy|^4+|\oH|^2}}\left(7-8\frac{|\uy|^4}{|\oH|^2}\right)+64|\uy|^4\left(\frac{1+\frac{|\uy|^4}{|\oH|^2}}{46|\uy|^4+|\oH|^2}\right)
    \leq \frac{|\uy|^4}{{46|\uy|^4+|\oH|^2}}\left(148-\frac{32|\uy|^4}{|\oH|^2}\right).
\end{align*}
We claim that \[ \frac{|\uy|^4}{{46|\uy|^4+|\oH|^2}}\left(148-\frac{32|\uy|^4}{|\oH|^2}\right)=\frac{\tfrac{|\uy|^4}{|\oH|^2}}{46\tfrac{|\uy|^4}{|\oH|^2}+1}\left(148-\frac{32|\uy|^4}{|\oH|^2}\right)\leq 2.81.\]
This can be seen by optimizing the function \[f(u)=\frac{u(148-32u)}{46u+1}\]
on the interval $[0,\infty].$ The equation $f'(u)=0$ has a unique solution in this interval, namely $u_0=\frac{3\sqrt{95}}{92}-\frac{1}{46}$. Further, $f''(u)=-\frac{13680}{(46u+1)^3}<0$ for $u\in [0,\infty)$ and thus $u_0$ is a point of maximum on the interval. Evaluating, we get
\[f(u_0)=\frac{2}{529}(859-12\sqrt{95})\sim 2.8054\]
Thus $|D+\lambda^{-1}B^{\intercal}B|\geq 3-2.81=0.19.$ We conclude that the rank of the cinematic curvature matrix is $2n$, thus verifying condition \eqref{eq:rank-curv eq} in the intermediate setting and finishing the proof of Proposition \ref{prop kernel est eq}.

\subsection{Proof of Proposition \ref{prop:L2osc-est} at the Poles}
\label{subsec L2 np}
We end this section by briefly describing how the calculation in $\S \ref{subsubsec IM L2}$ also yields the fixed time $L^2$ estimate for $T^k$ at the poles. In \S \ref{subsubsec IM L2}, we showed that the determinant of Hessian of the the phase function is bounded from below by a constant times the determinant of the matrix $X$ given by
\[X=4\uy\uy^{\intercal}+2|\uy|^2I-\frac{\oH}{2} J-\frac{|\uy|^2}{\oH}\uy(J\uy)^{\intercal}.\]
Here $\oH$ is defined uniquely such that $|\uy|^4+|\oH|^2=t^4$. At the poles, since $|\uy|=0$ and $\oH=t^2$, it follows that $X=-\frac{t^2}{2}J$ and 
\[|\det X|=\frac{t^2}{2} |\det J| \geq \frac{1}{2}\]
for $t\in [1,2]$ and $J$ satisfying the condition $J^2=-I$. 
Thus the mixed Hessian of the phase function has rank $2n+1$ at the poles as well, which implies the fixed time estimate \eqref{fix time Tkest} for the corresponding $T^k$ 
via H\"ormander's $L^2$ estimate. Observe that this argument works directly without using the Taylor's expansion \eqref{eq G Taylor exp} of the phase function.
\begin{rem}
We also note that arguing directly as above would not work when calculating the rank of the cinematic curvature matrix $\mathscr{C}$ around the poles. As is clear from \eqref{eq cin matrix IM}, $\mathscr{C}$ has rank $0$ when $|\uy|=0$. Thus the scaling argument of \S \ref{sec scaling} is required to understand the behaviour of the K\'oranyi sphere around the poles. 
\end{rem}

\section{Necessary Conditions}
\label{sec counter eg}
We provide five counter-examples, corresponding to each edge of the quadrilateral $\mathcal{R}$ in Theorem \ref{thm:max full} and one for the point $Q_2$, which show the necessity of all  the conditions 
in Theorem \ref{thm:max full}. 
They are suitable modifications of the examples in \cites{SchlagSogge1997, AHRS, RoosSeeger} for the Euclidean case, which were in turn adapted from standard examples for spherical means and maximal functions. The first four counter-examples are similar to the corresponding ones in \cites{roos2021lebesgue} for the maximal function associated to codimension two spheres on the Heisenberg groups. The fifth one is new and can be viewed as the replacement of a standard Knapp type example. 
It has the interesting feature that the dimensions of the set associated to the input test function are not equal to the dimensions of the set on which the output function (under the action of the operator) is large. This is quite unlike what is observed for the Euclidean Knapp examples.

We also briefly discuss the counter-examples which imply the necessary condition in Theorem \ref{thm single avg} for the single average and which can be easily inferred from the counter-examples for the maximal operator.

Recall that for $x=(\ux,\ox)\in \bbH^n$, the Kor\'anyi norm of $x$ is defined to be
\[|x|_K:=\left(|\ux|^4+|\ox|^2\right)^{\frac{1}{4}}.\]

\subsection{The line connecting \texorpdfstring{$Q_1$ and $Q_2$}{Q1 and Q2}}
\label{sec:Q1Q2}
This is the necessary condition $p\leq q$ imposed by translation invariance and noncompactness of the group $\bbH^n$ (see \cite{hormander1960}  for the analogous argument in the Euclidean case).

\subsection{\texorpdfstring{The line connecting $Q_2$ and $Q_3$}{The line connecting Q2 and Q3}} \label{sec:Q2Q3}
Let $B_{\delta}$ be the ball of radius $\delta$ centered at the origin. Let $f_{\delta}$ be the characteristic function of $B_{100\delta}$. Then \[\|f_{\delta}\|_p\approx \delta^{(2n+1)/p}.\] 
For $1\le t\le 2$ we consider the set
\[ R:= \{x=(\ubar{x}, \bar x) :1\leq |x|_K\leq 2\}.\]
Then $|R|\gc 1 $.
Let $\Sigma_{x}= \left\{\om=(\uom,\oom): |\om|_K=1, \left|\left(\tfrac{\ux}{|x|_K}, \tfrac{\ox}{|x|_K^2}\right)-(\uom,\oom)\right|\le \delta/4\right\}$, which has surface measure $\approx \delta^{2n}$.

For $x\in R$ and $\om \in \Sigma_{x}$, we set $t=|x|_K$. Then  \[|\ubar x-t\uom |\leq t|t^{-1}\ux-\uom|\le \delta\] and 
\[|\bar x-t^2\oom- \tfrac{t}{2} \ubar x^\intercal J\uom| \le 
t^2|t^{-2}\ox-\oom|+\frac{1}{2}|\ux^{\intercal}J(\ux-t\uom)|
\le 3\delta\]
(here we have used the skew symmetry of the $J$). We get 
\[f_{\delta}*\mu_t(\ubar{x},\bar{x})=\int_{|\om|_K=1} f_{\delta}(\ubar{x}-t{\uom},\bar{x}-t^2\oom-\tfrac{t}{2}\ubar{x}^\intercal\! J{\uom})\,d\mu({\om})\gtrsim \delta^{2n}\]
for  $x\in R$ and $t=|x|_K$. 
Passing to the maximal operator yields the inequality 
\[\delta^{2n}\lesssim \delta^{(2n+1)/p},\]
and consequently, the necessary condition
\begin{equation*}
\label{ball ineq}
2n\geq \frac{2n+1}{p},  
\end{equation*}
that is, $(1/p,1/q)$ lies on or above the line connecting $Q_2$ and $Q_3$.

\subsection{\texorpdfstring{The point $Q_2$}{The point Q2}} 
For $p=p_2:=\frac{2n+1}{2n}$ the $L^p\to L^p$ bound fails. Here one uses a modification of Stein's example \cite{SteinPNAS1976} for the Euclidean spherical maximal function.
One considers the function $f_\alpha$ defined by $f_\alpha(\ubar v, \overline{v})= |v|^{-\frac{2n+1}{p_2}} |\log |v||^{-\alpha}  $ for $|\ubar v|\leq 1/2$, $|\overline(v) |\leq 1$ which belongs to $L^{p_2}$ for $\alpha>1/p_2$. One finds that if $t(x)=|x|_K$ then for 
$\alpha<1$ the integrals $f*\sigma_{t(x)}(x)$ are $\infty$ on a set of positive measure. 

\subsection{\texorpdfstring{The line connecting $Q_1$ and $Q_4$}{The line connecting Q1 and Q4}} \label{sec:Q1Q4} For this line we just use the counterexample for the individual averaging operator, using it to bound the maximal function from below. 
Given $t\in [1,2]$, let $g_{\delta,t}$ be the characteristic function of the set
 $\{y=(\ubar{y},\bar{y}): | |y|_K-t|\le 100\delta\}$. Thus  $\|g_{\delta,t}\|_p\sim \delta^{1/p}. $ 

Let  $x=(\ubar{x},\bar{x})$ be such that  $|\ubar x|\le  \delta$ and $|\bar x| \le  \delta$. For any $\om=(\uom,\oom)$ with $|\om|_K=1$, we have $t|\ubar{x}^\intercal J{\om}|\lesssim 2\delta$. Furthermore,
\begin{align*}
    &|(\ux-t\uom, \ox-t^2\oom-\tfrac{t}{2}\ux^{\intercal}J\uom|_K^4-t^4\\
    &=|\ux-t\uom|^4+|\ox-t^2\oom-\tfrac{t}{2}\ux^{\intercal}J\uom|^2-t^4\\
    &\leq |\ux|^4+4t^2|\ux|^2|\uom|^2+2t^2|\ux|^2|\uom|^2+2t|\ux|^3|\uom|+2t^3|\ux||\uom|^3+t^4|\uom|^4\\
    &+|\ox|^2+t^4|\oom|^2+\tfrac{t^2}{4}|\ubar{x}^\intercal J{\om}|^2+2t^2|\ox||\oom|+t^3|\oom||\ubar{x}^\intercal J{\om}|+t|\ox||\ux^\intercal J\uom|-t^4\\
    &\lesssim \delta+t^4|\om|_K^4-t^4=\delta.
\end{align*}
Thus \[\left||(\ux-t\uom, \ox-t^2\oom-\tfrac{t}{2}\ux^{\intercal}J\uom|_K-t\right|\lesssim \frac{\delta}{t^3}\lesssim \delta,\]
implying that $|g_{\delta,t}*\sigma_t(x)|\gtrsim 1$. 
This yields the inequality $\delta^{(2n+1)/q}\le \delta^{1/p} $ which leads to the necessary condition
\begin{equation*}
\label{scaling ineq}
\frac {1}{q} \ge \frac{1}{2n+1} \frac 1p,   
\end{equation*} that is, $(1/p,1/q)$ lies on or above the line connecting $Q_1$ and $Q_4$.

\subsection{\texorpdfstring{The line connecting $Q_3$ and $Q_4$}{The line connecting Q3 and Q4, m=1}} \label{sec:Q3Q4}
Let $f_{\delta}$ be the characteristic function of the set 
\[\{y=(\uy,\oy):|\uy|\lesssim \delta^{1/4}, |\oy|\lesssim \delta\}.\]
Then $\|f_{\delta}\|_p\sim \delta^{\frac{2n}{4}+1}=\delta^{\frac{n}{2}+1}$.

For $t\in [1,2]$, let $R_{\delta,t}$ denote the set
\[R_{\delta,t}=\{x=(\ux,\ox):|\ux|\lesssim \delta^{\frac{3}{4}}, |\ox-t^2|\lesssim \delta\}.\]
Then $|R_{\delta,t}|\gtrsim \delta^{\frac{3n}{2}+1}$. Finally, let $\Sigma_{\delta,t}$ denote the set
\[\Sigma_{\delta,t}:=\{\om=(\uom,\oom):|\om|_K=1, |\uom|\lesssim \delta^{\frac{1}{4}}, \oom\geq 0\}.\]
The above condition on $|\uom|$ implies that $1-\oom\lesssim \delta$, for
\[1-\oom=\frac{1-|\oom|^2}{1+\oom}\leq 1-|\oom|^2=|\uom|^4\lesssim \delta.\]
Further $|\Sigma_{\delta,t}|\sim \delta^{\frac{2n}{4}}=\delta^{\frac{n}{2}}$ and we have

\begin{align*}
    |\ux-t\uom|&\lesssim |\ux|+t|\uom|\lesssim \delta^{\frac{3}{4}}+\delta^{\frac{1}{4}}\lesssim \delta^{\frac{1}{4}},\\
    |\ox-t^2\oom-\tfrac{t}{2}\ux^{\intercal}J\uom|&\lesssim |\ox-t^2|+t^2|(1-\oom)|+\frac{t}{2}|\ux^{\intercal}J\uom|\lesssim \delta+t^2\delta+\frac{t}{2}\delta^{\frac{3}{4}}\delta^{\frac{1}{4}}\lesssim \delta.
\end{align*}
As a consequence
\[f_{\delta}*\mu_t(\ubar{x},\bar{x})=\int_{|\om|_K=1} f_{\delta}(\ubar{x}-t{\om},\bar{x}-t^2\oom-\tfrac{t}{2}\ubar{x}^\intercal\! J{\om})\,d\mu({\om})\gtrsim \delta^{\frac{n}{2}}\]
for  $x\in R_{\delta,t}$. 
Passing to the maximal operator yields the inequality 
$\delta^{\frac{n}{2}}\delta^{(\frac{3n}{2}+1-1)\frac{1}{q}}\lesssim \delta^{(\frac{n}{2}+1)\frac{1}{p}}$ which leads to the necessary condition
\begin{equation*}
    \label{eq Q3Q4}
    n+\frac{3n}{q}\geq\frac{n+2}{p},
\end{equation*}
that is, $(1/p,1/q)$ lies on or above the line connecting $Q_3$ and $Q_4$.
\subsection{Necessary conditions for the averaging operator}
\label{sec: counter eg avg}
The necessary condition $p\geq q$ from \S\ref{sec:Q1Q2} corresponds to the line joining $(0,0)$ and $(1,1)$, while \S\ref{sec:Q1Q4} implies that $(1/p,1/q)$ lies on or above the line connecting $(0,0)$ and $(\frac{2n+1}{2n+2}, \frac{1}{2n+2})$.

For the line connecting $(\frac{2n+1}{2n+2}, \frac{1}{2n+2})$ and $(1,1)$, we use the same example as in \S\ref{sec:Q2Q3} but for a single average. More precisely, we set $t=1$ and test the output not on $R$ but on the set
\[ R_\delta:= \{x=(\ubar{x}, \bar x) :||x|_K-1|\leq \delta\}\]
with $|R_\delta|\sim \delta $.
As in \S\ref{sec:Q2Q3}, let $\Sigma_{x}= \left\{\om=(\uom,\oom): |\om|_K=1, \left|\left(\frac{\ux}{|x|_K}, \frac{\ox}{|x|_K^2}\right)-(\uom,\oom)\right|\le \delta/4\right\}$ with surface measure $\approx \delta^{2n}$.

For $x\in R_\delta$ and $\om \in \Sigma_{x}$, we have  \[|\ubar x-\uom |\leq |x|_K\left|\frac{\ux}{|x|_K}-\uom\right|+|\uom|\left||x|_K-1\right|\leq 2\delta\] and 
\[|\bar x-\oom- \tfrac{1}{2}\ubar x^\intercal J\uom| \le 
|x|_K^2\left|\frac{\ox}{|x|_K^2}-\oom\right|+|\oom|\left||x|_K^2-1\right|+\frac{1}{2}|\ux^{\intercal}J(\ux-t\uom)|
\le 4\delta.\]
Here we have used the skew symmetry of the $J$. We get 
\[f_{\delta}*\mu(\ubar{x},\bar{x})=\int_{|\om|_K=1} f_{\delta}(\ubar{x}-{\om},\bar{x}-\oom-\tfrac{1}{2}\ubar{x}^\intercal J{\om})\,d\mu({\om})\gtrsim \delta^{2n}\]
for  $x\in R_\delta$. This yields the inequality 
\[\delta^{2n}\delta^{1/q}\lesssim \delta^{(2n+1)/p},\]
and consequently, the necessary condition
\begin{equation}
2n+\frac{1}{q}\geq \frac{2n+1}{p},  
\end{equation}
that is, $(1/p,1/q)$ lies on or above the line connecting $(\frac{2n+1}{2n+2}, \frac{1}{2n+2})$ and $(1,1)$.

\section{Sparse Bounds}
\label{sec further}
As mentioned in the introduction, the principal goal  of  \cite{GangulyThangavelu}  was to derive   for the lacunary maximal operator $\fM^{\text{lac}}:=\sup_{k\geq 0}|\mathcal{A}_{2^k}(x)|$ an inequality of the form \Be\label{sparse-bound lac}  \int_{\bbH^n} \fM^{\text{lac}} f(x) w(x) dx \le C\sup\big\{  \La_{\cS, p_1,p_2} (f,w) :\,{\cS\, \mathrm{sparse}} \big\},
\Ee
where the supremum is taken over {\it sparse families } of nonisotropic Heisenberg cubes 
(see \cite{GangulyThangavelu} for precise definitions and constructions) and  the sparse form $\La_{\fS,p_1,p_2} $ is given by 
\[\La_{\cS,p_1,p_2} (f,w) = \sum_{S\in \cS} |S| \Big(\frac{1}{|S|} \int|f|^{p_1}\Big)^{1/p_1} \Big(\frac{1}{|S|} \int_S |f|^{p_2} \Big)^{1/p_2}.\]
By using the sharp $L^p\to L^q$ bounds in Theorem \ref{thm single avg} and by appealing to the arguments in \cites{BagchiHaitRoncalThangavelu, GangulyThangavelu}, we can prove the following theorem.
\begin{thm}
\label{thm sparse lac}
The sparse bound  \eqref{sparse-bound lac} 
holds  if  $(1/p_1, 1-1/p_2)$ lies in the interior of the triangle with vertices $(0,0), (1,1)$ and $(\tfrac{2n}{2n+1}, \frac 1{2n+1})$.
\end{thm}
The above result is sharp up to the boundary and improves upon the main result in \cite{GangulyThangavelu}.

Similarly, by applying the reasoning in \cites{GangulyThangavelu, BagchiHaitRoncalThangavelu} and using the $L^p\to L^q$ bounds in Theorem \ref{thm:max full}, we can prove the following result which is new and also sharp up to the boundary.
\begin{thm}
\label{thm sparse full}
Let $\mathfrak{M}f(x):=\sum_{t>0}|\mathcal{A}_t f(x)|$ be the K\'oranyi global maximal function.
The sparse bound \[\label{sparse-bound full}  \int_{\bbH^n} \fM f(x) w(x) dx \le C\sup\big\{  \La_{\cS, p_1,p_2} (f,w) :\,{\cS\, \mathrm{sparse}} \big\}
\]
holds whenever $(1/p_1, 1-1/p_2)$ lies in the interior of the quadrilateral $\mathcal{R}$ in 
\eqref{quadrilateral} (or on the open line segment $Q_1Q_2$).
\end{thm}

The proof of sparse bounds for the global maximal operator by using $L^p\to L^q$ estimates for localized maximal functions was pioneered by Lacey \cite{laceyJdA19} in his work on the Euclidean spherical maximal function.  
The recent papers \cite{BeltranRoosSeeger} and \cite{conde2021metric} give very general results about this correspondence for Euclidean spaces and spaces of homogeneous type, respectively.

\section{Appendix: Translation Invariance}
\label{section trans invar}
Let $\Theta:\bbR^{2n+1}\times\bbR^{2n+1}\to \bbR^{2n+1}$ denote the Heisenberg group translation map
\begin{equation*}
    \label{eq H law}
    \Theta(x,y):=
    \left(\ux-\uy,\ox-\oy+\tfrac{1}{2}\ux^{\intercal}J\uy\right).
\end{equation*}
Let $\Phi$ be a smooth real valued function on $\bbH^n\times[1,2]$. For our purpose, we shall think of $\Phi$ as a function from $\bbR^{2n+1}\times [1,2]\to\bbR$.  Let $\Tilde{\Phi}:\bbR^{2n+1}\times[1,2]\times\bbR^{2n+1}\to\bbR$ be defined as 
\begin{equation}
    \label{eq gen func defn}
    \Tilde{\Phi}(x,t,y):=\Phi(\Theta(x, y), t)=\Phi(\ux-\uy,\ox-\oy+\tfrac{1}{2}\ux^{\intercal}J\uy,t).
\end{equation}
We prove the following lemma which says that the rank of the curvature matrices associated to the oscillatory integral operator corresponding to the phase function $\Tilde{\Phi}$ is invariant under group translation.
\begin{lemma}
\label{lem trans invar}
Let $\Tilde{\Phi}$ be as defined above and let $t\in [1,2]$.

(i) \[\rank\Tilde{\Phi}''_{x,y}(x, t, y)=\rank\Tilde{\Phi}''_{x,y}(0,t,\Theta(x,y).\]

(ii) Let $\mathscr{C}$ be the cinematic curvature matrix for $\Tilde{\Phi}$ as defined in \eqref{eq curv matrix eq} with respect to a unit normal vector $N$. Then
\[\rank \mathscr{C}(x,t,y)=\rank \mathscr{C}(0,t,\Theta(x,y)).\]
\end{lemma}
\begin{proof}
Differentiating \eqref{eq gen func defn} and using the chain rule, we have \[\nabla_{(x,t)}\Tilde{\Phi}(x,t,y)=\nabla \Phi\big|_{(\Theta(x,y),t)}\cdot\begin{pmatrix}
\Theta_{(x,t)}'& & 0\\
0& & 1
\end{pmatrix}=\nabla \Phi\big|_{(\Theta(x,y),t)}\cdot\begin{pmatrix}
I_{2n}& & 0& & 0\\
\tfrac{1}{2}(J\uy)^{\intercal}& & 1& & 0\\
0 & & 0 & & 1
\end{pmatrix}.\]
Let $d=2n+1$. Expanding the matrix product on the right gives
\[\nabla_{(x,t)}\Tilde{\Phi}(x,t,y)=\left(\nabla_{2n}\Phi+\tfrac{1}{2}\Phi'_{d}J\uy,\Phi'_{d},\Phi'_{t}\right)\big|_{(\Theta(x,y),t)}.\]
Here $\nabla_{2n}\Phi$ denotes the partial derivative of $\Phi$ with respect to the first $2n$ coordinates, while $\Phi'_{d}, \Phi'_{t}$ denote its derivatives with respect to the $2n+1$-th and the last (time) coordinates respectively.

Another application of chain rule yields 
\begin{equation}
    \label{eq first deriv cov}
    \Tilde{\Phi}''_{(x,t),y}(x,t,y)=-\begin{pmatrix}
D_{2n}^2\Phi+\tfrac{1}{2}\nabla_{2n}\Phi'_{d}(J\uy)^{\intercal}+\tfrac{1}{2}\Phi_{d}'J& &\nabla_{2n}\Phi'_{d}+\tfrac{1}{2}\Phi''_{dd}J\uy\\
\nabla_{2n}\Phi'_{d}& &\Phi''_{dd}\\
\nabla_{2n}\Phi'_{t}& &\Phi''_{tt}
\end{pmatrix}\Bigg|_{(\Theta(x,y),t)}\begin{pmatrix}
I_{2n}& & 0\\
\tfrac{1}{2}(J\ux)^{\intercal}& & 1
\end{pmatrix}.
\end{equation}
Here $D_{2n}^2$ denotes the Hessian of $\Phi$ with respect to the first $2n$ coordinates. In particular, for a fixed $t$, the mixed Hessian of $\Tilde{\Phi}$ at $(x,y)$ is given by 
\begin{equation*}
\Tilde{\Phi}''_{x,y}(x, t, y)=-\begin{pmatrix}
D_{2n}^2\Phi+\tfrac{1}{2}\nabla_{2n}\Phi'_{d}(J\uy)^{\intercal}+\tfrac{1}{2}\Phi_{d}'J& &\nabla_{2n}\Phi'_{d}+\tfrac{1}{2}\Phi''_{dd}J\uy\\
\nabla_{2n}\Phi'_{d}& &\Phi''_{dd}
\end{pmatrix}\bigg|_{(\Theta(x,y),t)}\begin{pmatrix}
I_{2n}& & 0\\
\tfrac{1}{2}(J\ux)^{\intercal}& & 1
\end{pmatrix}.    
\end{equation*}

Since the second matrix in the product on the right hand side has determinant equal to one, the rank of $\Tilde{\Phi}''_{x,y}(x,t,y)$ is the same as the rank of 
\begin{align*}
&-\begin{pmatrix}
D_{2n}^2\Phi+\tfrac{1}{2}\nabla_{2n}\Phi'_{d}(J\uy)^{\intercal}+\tfrac{1}{2}\Phi_{d}'J& &\nabla_{2n}\Phi'_{d}+\tfrac{1}{2}\Phi''_{dd}J\uy\\
\nabla_{2n}\Phi'_{d}& &\Phi''_{dd}
\end{pmatrix}\bigg|_{(\Theta(x,y),t)}\\
&=-\begin{pmatrix}
D_{2n}^2\Phi+\tfrac{1}{2}\nabla_{2n}\Phi'_{d}(J\uy)^{\intercal}+\tfrac{1}{2}\Phi_{d}'J& &\nabla_{2n}\Phi'_{d}+\tfrac{1}{2}\Phi''_{dd}J\uy\\
\nabla_{2n}\Phi'_{d}& &\Phi''_{dd}
\end{pmatrix}\bigg|_{(\Theta(x,y),t)}\begin{pmatrix}
I_{2n}& & 0\\
0& & 1
\end{pmatrix}\\
&=\Tilde{\Phi}''_{x,y}(0,t,\Theta(x,y)).   
\end{align*}
This proves the first part. For the second part, we consider the surface \[\Sigma_{(x,t)}:=\nabla_{(x,t)}\Tilde{\Phi}(x,t,y).\] The tangent space of $\Sigma_{(x,t)}$ at a point $y$ is spanned by the column vectors of the matrix $\Tilde{\Phi}''_{(x,t),y}(x,t,y)$. Observe that \eqref{eq first deriv cov} implies that the column vectors of the matrix \[\begin{pmatrix}
D_{2n}^2\Phi+\tfrac{1}{2}\nabla_{2n}\Phi'_{d}(J\uy)^{\intercal}+\tfrac{1}{2}\Phi_{2n+1}'J& &\nabla_{2n}\Phi'_{d}+\tfrac{1}{2}\Phi''_{dd}J\uy\\
\nabla_{2n}\Phi'_{d}& &\Phi''_{dd}\\
\nabla_{2n}\Phi'_{t}& &\Phi''_{tt}
\end{pmatrix}\Bigg|_{(\Theta(x,y),t)}.\]
also form a basis for the tangent space to $\Sigma_{(x,t)}$ at $y$.
For $1\leq j\leq 2n+1$, let $\Xi_j$ denote the $j$-th column vector of the matrix above, and let $N\in \bbR^{2n+2}$ be a unit vector normal to $\Sigma_{(x,t)}$, so that $\langle N, \Xi_j\rangle=0$ for $1\leq j\leq 2n+1$.

Recall that the cinematic curvature matrix $\mathscr{C}(x,t,y)$ with respect to $N$ is a $(2n+1)\times (2n+1)$ dimensional matrix with the $(i,j)$-th entry given by $\left\langle N, \frac{\partial}{\partial y_i}\Xi_j(\Theta(x,y),t)\right\rangle$. We use subscript $i$ with $i\in \{1, \ldots, 2n+1\}$ to denote partial derivative in the direction of $e_i$. By the chain rule, the $i$-th row of $\mathscr{C}(x,t,y)$ is given by
\begin{equation*}
    N^{\intercal}\cdot\begin{pmatrix}
\mathscr{A}& &\mathscr{B}\\
\nabla_{2n}\Phi''_{di}& &\Phi'''_{ddi}\\
\nabla_{2n}\Phi''_{ti}& &\Phi'''_{tti}
\end{pmatrix}\Bigg|_{(\Theta(x,y),t)}
\begin{pmatrix}
I_{2n}& & 0\\
\tfrac{1}{2}(J\ux)^{\intercal}& & 1
\end{pmatrix},
\end{equation*}
with $\mathscr{A}:=D_{2n}^2\Phi'_i+\tfrac{1}{2}\nabla_{2n}\Phi''_{di}(J\uy)^{\intercal}+\tfrac{1}{2}\nabla_{2n}\Phi'_{d}(Je_i)^{\intercal}+\tfrac{1}{2}\Phi_{di}''J$ and $\mathscr{B}:=\nabla_{2n}\Phi''_{di}+\tfrac{1}{2}\Phi'''_{ddi}J\uy+\tfrac{1}{2}\Phi''_{dd}Je_i$.

Observe that 
\[
 N^{\intercal}\cdot\begin{pmatrix}
\mathscr{A}& &\mathscr{B}\\
\nabla_{2n}\Phi''_{di}& &\Phi'''_{ddi}\\
\nabla_{2n}\Phi''_{ti}& &\Phi'''_{tti}
\end{pmatrix}\Bigg|_{(\Theta(x,y),t)}   
= N^{\intercal}\cdot\begin{pmatrix}
\mathscr{A}& &\mathscr{B}\\
\nabla_{2n}\Phi''_{di}& &\Phi'''_{ddi}\\
\nabla_{2n}\Phi''_{ti}& &\Phi'''_{tti}
\end{pmatrix}\Bigg|_{(\Theta(x,y),t)}\begin{pmatrix}
I_{2n}& & 0\\
0& & 1
\end{pmatrix},
\]
which is the $i-$th row of the matrix $\mathscr{C}(0,t,\Theta(x,y))$. Thus we have
\[\mathscr{C}(x,t,y)=\mathscr{C}(0,t,\Theta(x,y))\begin{pmatrix}
I_{2n}& & 0\\
\tfrac{1}{2}(J\ux)^{\intercal}& & 1
\end{pmatrix}.\]
This clearly implies that the rank of the cinematic curvature matrix for $\Tilde{\Phi}$ at $(x,t,y)$ is equal to the cinematic curvature matrix rank at $(0,t,\Theta(x,y)).$
\end{proof}
\bibliography{ref}
\end{document}